\documentclass[12pt,a4paper,oneside]{article}
\usepackage{amsmath,amstext,amssymb,amsopn,amsthm,color}
\usepackage{graphicx}

\definecolor{orange}{RGB}{255,127,0}

\theoremstyle{plain}
\newtheorem{thm}{Theorem}

\newtheorem{lem}[thm]{Lemma}
\newtheorem{prop}[thm]{Proposition}

\theoremstyle{definition}

\theoremstyle{remark}
\newtheorem*{rem*}{Remark}
\newtheorem{rem}{Remark}

\newcommand{\R}{\mathbb{R}}

\newcommand{\N}{\mathbb{N}}

\newcommand{\C}{\mathbb{C}}

\renewcommand{\H}{\mathbb{H}}

\renewcommand{\leq}{\leqslant}
\renewcommand{\le}{\leq}
\renewcommand{\geq}{\geqslant}
\renewcommand{\ge}{\geq}
\newcommand{\qmu}{q_{\,x}^{(-\mu)}(t)}
\newcommand{\Amuplus}{A_{x}^{(\mu)}}
\newcommand{\Amu}{A_{x}^{(-\mu)}}

\def\({\left(}
\def\){\right)}
\def\[{\left[}
\def\]{\right]}
\def\<{\langle}
\def\>{\rangle}

\title{Hitting times of Bessel processes\footnotetext{2010 MS Classification:
    Primary 60J65; Secondary 60J60.
    {\it Key words and phrases}: Bessel process, first hitting time, geometric Brownian motion, hyperbolic Brownian motion.  Research  supported by  Polish Ministry of Science and Higher Eduction 
    grant N N201 3731 36  }
    }
\author{T. Byczkowski, J. Ma\l{}ecki, M. Ryznar\\
 Institute of Mathematics and Computer Sciences,\\
  Wroc\l{}aw University of Technology, Poland}

\begin{document}
\maketitle
\begin{abstract}
Let $T_1^{(\mu)}$ be the first hitting time of the point $1$ by  the Bessel process with index $\mu\in \R$ starting  from $x>1$. Using an integral formula for the density $q_x^{(\mu)}(t)$ of $T_1^{(\mu)}$, obtained in \cite{BR:2006}, we prove sharp estimates of the density of $T_1^{(\mu)}$ which {exhibit} the dependence both on   time and space variables. Our result  provides   optimal {uniform} estimates for the density of the hitting time  of the unit ball by the  Brownian motion in $\R^n$, which improve existing bounds. Another  application is to provide  sharp estimates for the Poisson kernel for half-spaces for hyperbolic Brownian motion in real hyperbolic spaces.
 
\end{abstract}

\section{Introduction}
{\it Bessel processes} play a prominent  role both in the theory of Brownian motion (see \cite{McKean:1960} and  \cite{ItoMcKean:1974}) as well as in various theoretical and practical applications. The n-dimensional Bessel process appears quite naturally as the n-dimensional Euclidean norm of  Brownian motion; the more intriguing applications are related to the celebrated {Ray-Knight} theorems describing the behaviour of the local time of Brownian motion in terms of two-dimensional (quadratic) Bessel process (see \cite{RevuzYor:2005}, \cite{Yor:1992} or \cite{BorodinSalminen:2002}). There is also an intimate relation between Bessel processes and
{\it the geometric Brownian motion} \cite{Lamperti:1972} ({the so-called Lamperti's theorem} - see Preliminaries). Bessel processes also appear when representing  
important {\it jump L\'evy processes} by means of {\it traces} of some multidimensional diffusions ({\it Bessel-Brownian diffusions}); see \cite{Molcanov:1969}, \cite{BMR2:2010}. Another important application consists of the fact that {\it the hyperbolic Brownian motion}, i.e. the canonical diffusion in the real hyperbolic space, can be represented as subordination of the standard Brownian motion
via an exponential functional of geometric Brownian motion. Thus, from Lamperti's theorem, {\it the Poisson kernel of half-spaces} for hyperbolic Brownian motion can be represented via subordination of the standard Brownian motion by the hitting  time of a Bessel process (see \cite{BR:2006}). We exploit this relationship for obtaining the precise bounds of 
the Poisson kernel of half-spaces for hyperbolic Brownian motion with drift.

Our main goal is to study estimates of the density of the distribution of the first hitting time $T^{(\mu)}_{a}$ of a level $a>0$ by a Bessel process starting from $x>a$. Our approach is based on an integral  formula for the density of $T^{(\mu)}_{a}$, given in the paper 
\cite{BR:2006} {(see also \cite{BGS:2007})}. This formula, although quite complex, proves to be a very effective one.
 {We deal with  processes with non-positive indices}, however our conclusions are valid also for positive ones, since the processes are equivalent on the sigma-algebra generated by the process before hitting $0$.
 
  To formulate our results denote by
   $R^{(\mu)}=\{R_t^{(\mu)},t\geq 0\}$ a Bessel process with index $\mu\in\R$. We note by $\textbf{P}_x^{(\mu)}$ the probability law of a Bessel process $R^{(\mu)}$ with index $\mu$ on the canonical paths space with starting point $R^{(\mu)}_0 =x$, where $x>0$. 
Let us denote the first hitting time of the level $a> 0$ by a Bessel process with index $\mu$
\begin{eqnarray*}
   T^{(\mu)}_a = \inf\{t>0; R^{(\mu)}_t=a\}\/.
\end{eqnarray*} 
  Our main result is a sharp estimate of the density of the hitting distribution if $x>a$.  By the scaling property of Bessel processes it is enough to consider $x>a=1$. The distribution of $T^{(\mu)}_1$ is well known for $\mu=\pm1/2$ when both distributions have the  $1/2$-stable positive distribution with scale parameter $\lambda=x-1$, however incomplete for $\mu=1/2$. 
 Our upper and lower estimates are comparable in both time and space domains. We have the following:
\begin{center}
{\bf Uniform estimate for a density function of $T_1^{(\mu)}$.}
\end{center}
  For every $x>1$ and $t>0$ we have
  \begin{eqnarray*}
     \frac{\textbf{P}_x^{(\mu)}(T_1^{(\mu)}\in dt)}{dt} \approx (x-1)\left(\frac{1}{1+x^{2\mu}}\right)\frac{ {e^{-(x-1)^2/2t}}}{t^{3/2}} \frac{ x^{2|\mu|-1} }{t^{|\mu|-1/2}+ x^{|\mu|-1/2}}\/,\quad \mu\neq 0\/.
  \end{eqnarray*}
  Moreover, we have
  $$
    \frac{\textbf{P}_x^{(0)}(T_1^{(0)}\in dt)}{dt} \approx
  (x-1) e^{-(x-1)^2/2t} \frac{(x+t)^{1/2}}{x t^{3/2}} 
  \frac {1+\log x}{(1+\log (1+\frac tx))(1+\log (t+x))}\/.
  $$
Here $ f\approx g$ means that there exist strictly positive constants $c_1$ and $c_2$ depending only on $\mu$ such that $c_1\le f/g\le c_2$.
 
Our setup includes the case of the hitting distribution of a unit ball by a Brownian motion starting from the exterior of the ball.
To state this result,
let $\sigma^{(n)}$ be the first hitting time of a unit ball by $n$-dimensional Brownian motion $W^{(n)}=\{W_t^{(n)},t\geq 0\}$, i.e.
\begin{eqnarray*}
   \sigma^{(n)} = \inf\{t>0; |W^{(n)}_t|=1\}\/.
\end{eqnarray*}
\begin{center}
{\bf Uniform estimate for hitting the unit ball by Brownian motion.}
\end{center}
For $W^{(n)}_0 = x\in\R^n$ such that $|x|>1$ we have
\begin{eqnarray*}
 \frac{  P^{x}(\sigma^{(n)}\in dt)}{dt} \approx \frac{|x|-1}{|x|}\frac{ {e^{-(|x|-1)^2/2t}}}{t^{3/2}} \frac{ 1}{t^{(n-3)/2}+ |x|^{(n-3)/2}}\/, \quad n>2,
\end{eqnarray*}
for every $t>0$. Moreover, we have
\begin{eqnarray*}
  \frac{P^{x}(\sigma^{(2)}\in dt)}{dt} \approx  \frac{|x|-1}{|x|} e^{-(|x|-1)^2/2t} \frac{(|x|+t)^{1/2}}{ t^{3/2}} 
  \frac {1+\log |x|}{(1+\log (1+\frac t{|x|}))(1+\log (t+|x|))}\/.
\end{eqnarray*}
%

 To the best of our knowledge even for the planar Brownian motion our results are new and considerably complement existing results obtained in \cite{Grigoryan:2002} (see also \cite{Collete:2000}), where the estimates are {only} sharp in the  region $t \ge |x|^2$  with sufficiently large  starting point $x$. If  $t<|x|^2$ the bound obtained in \cite{Grigoryan:2002} has an  exponential term of the form $\exp\{-c_i\frac{|x|^2}{t}\}$ with different constants $c_1, c_2$   for the lower and the upper estimate, respectively. We remove this  obstacle and provide sharp estimates which are of the same order in the full range of $t$ and $x$. We also provide sharp estimates for the survival probability $
    \textbf{P}_x^{(\mu)}(t<T_1^{(\mu)}<\infty)$. The asymptotic result if $t\to \infty$ is due to Hunt \cite{Hunt:1956} in the case of the planar Brownian motion and Port \cite{Port:1969b} in the context of Brownian motion in higher dimensions. { Recently, a result about the asymtotic behaviour of the hitting density for the planar Brownian, when $t\to \infty$, was established by Uchiyama \cite{Uchiyama:2010}. His result gives a very accurate expansion  of the hitting density provided $ |x|^2$ is small relative to $t$, that is when the exponential term is negligable. When this is not true the error term of the expansion in \cite{Uchiyama:2010} may be much bigger than the leading term. Therefore our estimate in the case of the planar Brownian motion is much more accurate, when we assume that $t$ is not too large  with respect to $|x|^2$, since the impact of of the exponential term may be  significant and it is reflected in our estimates. Moreover, we give a very exact estimate of the density in the situation when $t$ is small relative to $|x|$ (see Lemma \ref{qt:estimate:zero} and Remark \ref{remark1}). }

The organization of the paper is as follows. After Preliminaries, in Section 3, we provide uniform estimates of the density function of the first hitting time $T^{(\mu)}_{1}$. This section is basic for further applications, which are collected in the next section. 
We first  provide the estimates of the {\it survival times} of a killed Bessel process and, finally, compute the precise bounds  of the Poisson kernel of a halfspace  for hyperbolic Brownian motion with drift.
Appendix contains various estimates of quantities involved in the basic formula for the density function of $T_1^{(\mu)}$, which, although quite laborious,  but at the same time, are indispensable ingredients of the proof of the main result.

Throughout the whole paper $ f\approx g$ means that there exists a strictly positive constant $c$  depending only on $\mu$ such that $c^{-1}g\le f \le c g$. If the comparability constant will also depend on some other parameters $\gamma_1, \gamma_2,\dots$ we will write $ f\stackrel c{\approx} g$, $c=c(\gamma_1, \gamma_2,\dots)$. Also in a string of inequalities a constant may change from line to line which might not be reflected in notation. Moreover we consistently do not exhibit dependence of $c$ on $\mu$ in inequalities of type $f\le c g $.

\section{Preliminaries}
\subsection{Modified Bessel functions}

Various formulas appearing throughout the paper are expressed in terms of {\it modified Bessel functions}  $I_{\vartheta}$ and $K_{\vartheta}$. For convenience we collect here basic information about these functions. 

The \textit{modified Bessel functions of the first} and \textit{second kind} are independent solutions to the \textit{modified Bessel differential equation}
\begin{eqnarray*}
  z^2y'' +  zy' - (\vartheta^2+z^2)y =0\/,
\end{eqnarray*}
where $\vartheta \in \R$. The Wronskian of the pair $\left\{K_\vartheta(z),I_\vartheta(z)\right\}$ is equal to 
\begin{eqnarray}
  \label{Wronskian}
  W\left\{K_\vartheta(z),I_\vartheta(z)\right\} = I_\vartheta(z) K_{\vartheta+1}(z)+I_{\vartheta+1}(z)K_\vartheta(z) = \frac{1}{z}\/.
\end{eqnarray}

In the sequel we will use the asymptotic behavior of $I_\vartheta$ and $K_\vartheta$ at zero as well as at infinity. For every $\vartheta \geq 0$ we have (see \cite{AbramowitzStegun:1972} 9.6.7 and 9.6.12)
\begin{eqnarray}
   \label{I_atzero}
   I_\vartheta(r) &=&  \frac{r^\vartheta}{2^\vartheta\Gamma(\vartheta+1)} +O(r^{\vartheta+2})\,,\quad r\to 0^{+}\/, \vartheta >0\/.
\end{eqnarray}
For $\vartheta>0$, we have (\cite{AbramowitzStegun:1972} 9.6.9 and 9.6.13)
 \begin{eqnarray}
   \label{K_atzero}
   K_\vartheta(r)&\cong& {\frac{2^{\vartheta-1}\Gamma(\vartheta)}{r^{\vartheta}}}\,,
   \quad r\to 0^+, \label{asympt_K_0}
 \end{eqnarray}
 where  $g(r) \cong f(r) $ means that the ratio of $g$ and $f$ tends to $1$. Moreover, in the case $\vartheta = 0$, we have (see \cite{AbramowitzStegun:1972} 9.6.13)
 \begin{eqnarray}
    \label{K0_atzero}
    K_0(r) = -\log\frac{r}{2}I_0(r) + O(1)\/,\quad r\to 0^{+}\/.
 \end{eqnarray}

 The behavior of $I_\vartheta$ and $K_\vartheta$ at infinity is described as follows (see \cite{AbramowitzStegun:1972} 9.7.1, 9.7.2)
\begin{eqnarray} 
\label{asymp_I_infty}
I_\vartheta(r)= {\frac{e^r}{\sqrt{2\pi r}}} (1+ O(1/r))\,,\quad r\to \infty\/,
\end{eqnarray}
\begin{eqnarray}
\label{asymp_K_infty}
K_\vartheta(z) = \sqrt{\frac{\pi}{2z}}\,e^{-z}(1+O(1/z))\/,\quad |z|\to \infty\/,
\end{eqnarray}
where the last equality is true for every complex $z$ such that $|arg z|<\frac{3}{2}\pi$.

\subsection{Bessel process and exponential functionals of Brownian motion}
In the following section we introduce notation and basic facts about Bessel processes. We follow the exposition given in \cite{MatsumotoYor:2005a} and \cite{MatsumotoYor:2005b}, where we refer the Reader for more details and deeper insight into the subject (see also \cite{RevuzYor:2005}). 

 We denote by $\textbf{P}_x^{(\mu)}$ the probability law of a Bessel process $R^{(\mu)}$ with index $\mu$ on the canonical paths space with starting point $R^{(\mu)}_0 =x$, where $x>0$. Let $\mathcal{F}^{(\mu)}_t=\sigma\{R^{(\mu)}_s,s\leq t\}$ be the filtration of the coordinate process $R^{(\mu)}_t$. The state space of $R^{(\mu)}$ depends on the value of $\mu$ and the boundary condition at zero. For simplicity, in the case $-1<\mu<0$ (then the point $0$ is non-singular), we impose killing condition on $0$. However, the exact boundary condition at $0$ is irrelevant from our point of view, because we will only consider the process $R^{(\mu)}$ up to the first hitting time of the strictly positive level. 

Let us denote the first hitting time of the level $a> 0$ by a Bessel process with index $\mu$
\begin{eqnarray*}
   T^{(\mu)}_a = \inf\{t>0; R^{(\mu)}_t=a\}\/.
\end{eqnarray*}
Observe that for $\mu\leq 0$ we have $T^{(\mu)}_a<\infty$ a.s. and $\textbf{P}^{(\mu)}_x(T^{(\mu)}_a=\infty)>0$ whenever $\mu>0$. Using the scaling property of Bessel processes, which is exactly the same as the scaling property of one-dimensional Brownian motion, we get  for every $b>0$ and $t>0$  
\begin{eqnarray*}
   {\textbf{P}_{bx}^{(\mu)}(T^{(\mu)}_{ba}<t) = \textbf{P}_{x}^{(\mu)}(b^2T^{(\mu)}_{a}<t)}\/, \quad x>a>0\/.
\end{eqnarray*}
Therefore, from now on we do assume that $a=1$ and $x>1$. We denote the density function of $T_1^{(\mu)}$ with respect to Lebesgue measure by $q_x^{(\mu)}$, i.e. 
\begin{eqnarray*}
   q_x^{(\mu)}(t) = \frac{\textbf{P}_{x}^{(\mu)}(T^{(\mu)}_{1}\in dt)}{dt}\/,\quad t>0\/,x>1\/.
\end{eqnarray*}

We have the absolute continuity property for the laws of the Bessel processes with different indices
\begin{eqnarray}
  \label{Bessel:AC}
  \left. \frac{d\textbf{P}^{(\mu)}_x}{d\textbf{P}^{(\nu)}_x}\right|_{\mathcal{F}^{(\nu)}_t} = \left(\frac{{R_t}}{x}\right)^{\mu-\nu}\exp\left(-\frac{\mu^2-\nu^2}{2}\int_0^t\frac{ds}{{(R_s)}^2}\right)\/,\quad {\textbf{P}^{(\nu)}_x}\textrm{ - a.s. on }\{T^{(\nu)}_0>t\}\/.
\end{eqnarray}
Here $T_0^{(\mu)}$ denotes the first hitting time of $0$ by $R^{(\mu)}$. If $\nu\geq 0$ then the condition $\{T^{(\nu)}_0>t\}$ can be omitted. In particular, for $\nu=-\mu$, where $\mu\geq 0$ we have
\begin{eqnarray*}
   \left. \frac{d\textbf{P}^{(\mu)}_x}{d\textbf{P}^{(-\mu)}_x}\right|_{\mathcal{F}^{(-\mu)}_t} = \left(\frac{{R_t}}{x}\right)^{2\mu}\/,\quad {\textbf{P}^{(-\mu)}_x}\textrm{ - a.s. on }\{T^{(-\mu)}_0>t\}\/.
\end{eqnarray*}
Consequently, for $\mu\geq 0$ and $x>1$ we get that
\begin{eqnarray}
   \label{BesselTime:indices}
   q_x^{(\mu)}(t) = \left(\frac1x\right)^{2\mu}\qmu\/.
\end{eqnarray}

We denote by $B=\{B_t,t\geq 0\}$ the one-dimensional Brownian motion starting from $0$ and by $B^{(\mu)}=\{B_t^{(\mu)}=B_t+\mu t,t\geq 0\}$ the Brownian motion with constant drift $\mu\in \R$. 
The process $X^{(\mu)} = \{x\exp(B^{(\mu)}_t),t\geq 0\}$ is called a \textit{geometric Brownian motion} or \textit{exponential Brownian motion} with drift $\mu\in\R$ starting from $x>0$. 

For $x>1$ let $\tau$ be the first exit time of the geometric Brownian motion with drift $\mu$ from the set $(1,\infty)$ 
\begin{eqnarray*}
   \tau = \inf\{s>0; x\exp(B_s+\mu s)=1\}\/.
\end{eqnarray*}
We have $\tau<\infty$ a.s. whenever $\mu\leq 0$ since then $\inf_{t\geq 0}B^{(\mu)}(t) = -\infty$.

For $x>0$ we consider the integral functional 
\begin{eqnarray*}
\Amuplus(t)= \int_0^t (X_s^{(\mu)})^2ds =  x^2\int_0^t\exp(2B_s+2\mu s)ds\/.
\end{eqnarray*}

The crucial fact which establishes the relation between Bessel processes, the integral functional $\Amuplus$ and the geometric Brownian motion is the Lamperti relation saying that there exists a Bessel process $R^{(\mu)}$ such that 
\begin{eqnarray*}
  x\exp(B_t^{(\mu)}) = R^{(\mu)}_{A^{(\mu)}_x(t)}\/,\quad t\geq 0\/.
\end{eqnarray*} 
Consequently, we get 
\begin{eqnarray}
   \label{Amu_Tmu}
   A^{(\mu)}_x(\tau) \stackrel{d}{=} T_1^{(\mu)}\/.
\end{eqnarray}


\subsection{Representation of hitting time density function}
We recall the result of \cite{BR:2006}, where the integral formula for the density of $\Amu(\tau)$ was given. According to (\ref{Amu_Tmu}), as immediate consequence, we obtain the formula for $\qmu$.  Note also that in the paper \cite{BR:2006} different normalizations of Brownian motion and different definition of geometric Brownian motion were used and consequently we have $\qmu =  {q_\mu}(t/2)/2$, where ${q_\mu}(t)$ is the density function considered in \cite{BR:2006}.

 
 \begin{thm}[{[Byczkowski, Ryznar 2006]}]\label{rep}
For $\mu \geq 0$ there is a function $w_\lambda$ such that
  \begin{eqnarray}\label{rep1}
    \qmu = \lambda \frac{e^{-\lambda^2/2t}}{ \sqrt{2\pi t} }
   \(\frac{x^{\mu-1/2}}{t} +
    \int_0^\infty \(e^{-\kappa/2t} - 1\) w_\lambda(v)dv \) \/,
    \end{eqnarray}
  where $\kappa = \kappa (v)=(\lambda+v)^2 - \lambda^2 = v(2\lambda+v)$,
  and $\lambda=x-1$.
%
\end{thm}
The function $w_\lambda$ appearing in the formulas is described in terms of the \textit{modified Bessel functions} $K_\mu$ and $I_\mu$. The function $K_\mu(z)$ extends  to an entire function when $\mu-1/2$ is an integer and has a holomorphic extension to $\C\setminus (-\infty,0]$ when $\mu-1/2$ is not an integer.
Denote the set of zeros of the function $K_\mu(z)$ by $Z=\{z_1,...,z_{k_\mu}\}$ (cf. \cite{Erdelyi:1954}, p. 62). Recall that $k_\mu=\mu-1/2$
when $\mu-1/2 \in \N$. For $\mu-1/2 \notin \N$, $k_\mu$ is the even number closest to $\mu-1/2$. The functions $K_\mu$ and $K_{\mu-1}$ have no common zeros.

The function $w_\lambda$ is defined as a sum of two functions
\begin{eqnarray}
  \label{wlambda:sum}
  w_\lambda(v)=w_{1,\/\lambda}(v)+w_{2,\/\lambda}(v)\/,
\end{eqnarray}  
where 
$$w_{1,\/\lambda}(v)= -\frac{x^\mu }{ \lambda} \sum_{i=1}^{k_\mu}
 \frac{z_i e^{\lambda z_i} K_\mu(xz_i) }{ K_{\mu-1}(z_i)} \, e^{z_i v}$$
and
\begin{eqnarray*}
 w_{2,\/\lambda}(v)= -\cos(\pi\mu) \frac{x^\mu }{ \lambda}
 \int_0^\infty \frac{
     I_\mu\(xu\)K_\mu(u)-I_\mu(u)K_\mu\(xu\)}{
\cos^2(\pi\mu) K_\mu^2(u)+(\pi I_\mu(u)+\sin(\pi \mu) K_\mu(u))^2} \,
 e^{-\lambda u} e^{-vu} \/u du\/.
\end{eqnarray*}
{Moreover the moments of $\kappa$ with respect to $w_\lambda(v)dv$ can be computed in the following way}
   \begin{eqnarray}\label{laplace01}
   x^{\mu-1/2} (\mu^2-1/4)/2x=\int_0^\infty w_\lambda(v) dv\/,
   \end{eqnarray}
   and, for $\mu>1/2$, we have
 \begin{eqnarray}\label{laplace02}
  2 x^{\mu-1/2} =\int_0^\infty \kappa w_\lambda(v) dv\/.
   \end{eqnarray}
   
We use also the following representation of $\qmu$ for $\mu\geq 1/2$ (see \cite{BR:2006} (24))   
\begin{eqnarray}
  \label{rep3}
  \qmu = \lambda\,\frac{e^{-\lambda^2/2t}}{\sqrt{2\pi t}}\int_0^\infty \left(e^{-\kappa/2t}-\sum_{0\leq j \leq l}(-1)^j\frac{1}{j!}\left(\frac{\kappa}{2t}\right)^j\right)w_\lambda(v)dv\/,
\end{eqnarray}
where $l=[\mu+1/2]$ if $\mu \notin \N$, and $l=\mu-1/2$ otherwise.


\section{Uniform estimates of hitting time density function}
Throughout of the rest of the paper we denote $\lambda=x-1$. The main result of the paper is the following uniform estimate for a density function of $T_1^{(\mu)}$.
\begin{thm}
  For every $x>1$ and $t>0$ we have
  \begin{eqnarray}
     \label{Bessel:hittingtime:estimates}
     q^{(\mu)}_x(t) \approx \lambda \left(\frac{1}{1+x^{2\mu}}\right)\frac{ {e^{-\lambda^2/2t}}}{t^{3/2}} \frac{ x^{2|\mu|-1} }{t^{|\mu|-1/2}+ x^{|\mu|-1/2}}\/,\quad \mu\neq 0\/.
  \end{eqnarray}
  Moreover, we have
    $$
    q^{(0)}_x(t) \approx
    \left\{\begin{array}{lc} \lambda \dfrac{e^{-\lambda^2/2t} }{x t }
  \dfrac {1+\log x}{(1+\log \frac tx)(1+\log t)}\/,& t>2x\/,\\
  {\lambda} \dfrac{e^{-\lambda^2/2t} }{x^{1/2} t^{3/2} }\/, & t\le 2x
  \end{array}\right.
  $$
  or equivalently
  $$
    q^{(0)}_x(t) \approx
  \lambda e^{-\lambda^2/2t} \frac{(x+t)^{1/2}}{x t^{3/2}} 
  \frac {1+\log x}{(1+\log (1+\frac tx))(1+\log (t+x))}\/.
  $$
\end{thm}
As corollary, putting $\mu=n/2-1$, we get the corresponding result for $n$-dimensional Brownian motion. 
\begin{thm}
Let $\sigma^{(n)}$ be the first hitting time of a unit ball by $n$-dimensional Brownian motion $W^{(n)}=\{W_t^{(n)},t\geq 0\}$, i.e.
\begin{eqnarray*}
   \sigma^{(n)} = \inf\{t>0; |W^{(n)}_t|=1\}\/.
\end{eqnarray*}
Then, for $W^{(n)}_0 = x\in\R^n$ such that $|x|>1$ we have
\begin{eqnarray*}
   \frac{P^{x}(\sigma^{(n)}\in dt)}{dt} \approx \frac{|x|-1}{|x|}\frac{ {e^{-(|x|-1)^2/2t}}}{t^{3/2}} \frac{ 1}{t^{(n-3)/2}+ |x|^{(n-3)/2}}\/, \quad n>2,
\end{eqnarray*}
for every $t>0$. Moreover, we have
\begin{eqnarray*}
  \frac{P^{x}(\sigma^{(2)}\in dt)}{dt} \approx  \frac{|x|-1}{|x|} e^{-(|x|-1)^2/2t} \frac{(|x|+t)^{1/2}}{ t^{3/2}} 
  \frac {1+\log |x|}{(1+\log (1+\frac t{|x|}))(1+\log (t+|x|))}
\end{eqnarray*}
\end{thm}

The proof of the main  theorem follows from Lemmas \ref{qt:estimate:inside}, \ref{qt:estimate:infty1}, \ref{qt:estimate:infty2}, \ref{qt:zero:estimate:infty}   given below. We use the crucial estimates of the function $w_\lambda(v)$ and its components $w_{1,\,\lambda}(v)$ and $w_{2,\,\lambda}(v)$ given in Appendix (see Lemmas \ref{w1:estimate:lemma}, \ref{w2:estimate:lemma} and \ref{w:muzero:estimate:lemma}).

  
The following lemma provides satisfactory estimates of $\qmu$ in the case when $t$ is small relative to $x$. 
{\begin{lem}  \label{qt:estimate:zero}
  We have the following expansion
  \begin{eqnarray*}
   \qmu=  \lambda\frac{e^{-\lambda^2/4t}}{(2\pi)^{1/2}t^{3/2}}x^{\mu-1/2} \left(1+\frac{1-4\mu^2}8\frac t{x} + E(t,x)\right),
    \end{eqnarray*}
  where the error term satisfies the following estimate
$$  |E(t,x)|\le  C \frac tx (\sqrt{t}\wedge \frac {t}\lambda).$$
   Moreover, for $0\le \mu<1/2$  we have
  \begin{eqnarray*}
  \lambda\frac{e^{-\lambda^2/2t}}{(2\pi)^{1/2}t^{3/2}}x^{\mu-1/2}\le   \qmu\le  \lambda\frac{e^{-\lambda^2/4t}}{(2\pi)^{1/2}t^{3/2}}x^{\mu-1/2} \left(1+\frac{1-4\mu^2}8\frac t{x} \right)
  \end{eqnarray*}
 for every $x>1, t>0$.
\end{lem} }
\begin{proof}
  { By the basic formula (\ref{rep1}), with application of (\ref{laplace01}), we obtain
   \begin{eqnarray*}
     \frac{(2\pi)^{1/2}t^{3/2}}{\lambda x^{\mu-1/2}}\,e^{\lambda^2/2t}\qmu &=& 1+\frac{t}{x^{\mu-1/2}}\int_0^\infty (e^{-\kappa/2t}-1)w_\lambda(v)dv\\
   & =& 1+\frac{1-4\mu^2}8\frac t{x} + \frac{t}{x^{\mu-1/2}}\int_0^\infty e^{-\kappa/2t}w_\lambda(v)dv\\
      & =& 1+\frac{1-4\mu^2}8\frac t{x} + E(x,t)
     \/.
   \end{eqnarray*} }

 {  Observe that, by Lemmas \ref{w1:estimate:lemma},  \ref{w2:estimate:lemma}
 and \ref{w:muzero:estimate:lemma},  $|w_\lambda(v)|\le C x^{\mu-3/2}$, which gives the following estimate 
   \begin{eqnarray*}
     |E(x,t)|= \left|\frac{t}{x^{\mu-1/2}}\int_0^\infty e^{-(\lambda v) /t}e^{- v^2 /2t}w_\lambda(v)dv\right|\leq
     C \frac tx\int_0^\infty e^{-(\lambda v) /t}e^{- v^2 /2t}dv\le C \frac tx (\sqrt{t}\wedge \frac {t}\lambda)
       \/.
   \end{eqnarray*} }
 This ends the proof of the first claim. For $\mu<1/2$ observe that 
   \begin{eqnarray*}
    0\le  \int_0^\infty (e^{-\kappa/2t}-1)w_\lambda(v)dv\le  -\int_0^\infty w_\lambda(v)dv = x^{\mu-1/2} (1/4-\mu^2)/2x.
   \end{eqnarray*} 
   Consequently,  we get
   \begin{eqnarray*}
     1\leq\frac{(2\pi)^{1/2}t^{3/2}x^{1/2}}{\lambda}\,e^{\lambda^2/4t}\qmu\leq\left(1+\frac{1-4\mu^2}8\frac t{x} \right)
   \end{eqnarray*}
   and this completes the proof of the second claim in the case $\mu<1/2$.
\end{proof}
{\begin{rem}\label{remark1} If $ 1<x<2$ then the absolute value of the error term is bounded by    $Ct^{3/2}$, while for $x>2$ it is bounded by $C(\frac tx)^2$. Observe also that the density $\qmu$ up to a multiplicative constant is close to the density of the hitting distribution of $1$ by the one-dimensional  Brownian motion  starting from $x$, when the fraction $\frac tx$ is small with the error precisely estimated by the above lemma. With some additional effort one can show that the error term estimate can not be improved. 
\end{rem}}
{
\begin{prop}
\label{Prop:asymp:inside}
For every $\mu\neq 0$ and $c>0$ we have
\begin{eqnarray*}
   \lim_{x/t\to c,\,x\to \infty} \dfrac{1+x^{2\mu}}{x^{|\mu|-1/2}}\frac{q_x^{(\mu)}(t)\sqrt{2\pi t}}{ e^{-\lambda^2/2t}} &=& \sqrt{\frac{\pi c}{2}} \frac{e^{-c}}{K_{|\mu|}(c)}\/.
\end{eqnarray*}
\end{prop}
\begin{proof}
   It is enough to show the above-given convergence for strictly negative indices. The general statement follows from (\ref{BesselTime:indices}). Now we assume that $\mu>0$ and consider $\qmu$. We define 
\begin{eqnarray*}
    w_\infty (v) = -\sum_{i=1}^{k_\mu} \frac{\sqrt{z_i}e^{-z_i}}{K_{\mu-1}(z_i)}e^{-vz_i}-\frac{\cos(\pi\mu)}{\sqrt{2\pi}}\int_0^\infty
    \frac{e^{u}K_\mu(u)e^{-vu}\sqrt{u}du}{\cos^2(\pi\mu) K_\mu^2(u)+(\pi I_\mu(u)+\sin(\pi \mu) K_\mu(u))^2} 
\end{eqnarray*}
for every $v>0$. If $k_\mu=0$ then the first sum is equal to zero. Using the asymptotic expansion (\ref{asymp_K_infty}) we easily see that
\begin{eqnarray*}
   \lim_{x\to\infty} \frac{\lambda}{x^{\mu-1/2}}w_{1,\,\lambda}(v) = -\sum_{i=1}^{k_\mu} \frac{\sqrt{z_i}e^{-z_i}}{K_{\mu-1}(z_i)}\,e^{-vz_i}\/.
\end{eqnarray*}
The relation (\ref{asymp_I_infty}) implies that $|I_\mu(u)|\leq c_2 \dfrac{e^u}{\sqrt{u}}$ and consequently
\begin{eqnarray*}
   \sqrt{x}|I\mu(xu)K_\mu(u)-I_\mu(u)K_\mu(xu)|e^{-xu}\leq  c_2 K_\mu(u) \frac{1}{\sqrt{u}}
\end{eqnarray*}
for every $u>0$. Moreover, using the estimates of $K_\vartheta$ and $I_\vartheta$ given in Preliminaries, we observe that the function
\begin{eqnarray*}
   f(u,v) = \frac{e^{u}K_\mu(u)e^{-vu}\sqrt{u}du}{\cos^2(\pi\mu) K_\mu^2(u)+(\pi I_\mu(u)+\sin(\pi \mu) K_\mu(u))^2}
\end{eqnarray*}
is bounded, as a function of $u$, by $c_3 e^{-(v+2)u}u^{3/2}$ on $[1,\infty)$ and by $c_3 u^{\mu+1/2}$ on $(0,1)$ and consequently is integrable on $(0,\infty)$. Using the dominated convergence theorem we get
\begin{eqnarray*}
   \lim_{x\to \infty} \frac{\lambda}{x^{\mu-1/2}}w_{2,\,\lambda}(v) = -\frac{\cos(\pi\mu)}{\sqrt{2\pi}}\int_0^\infty
    \frac{e^{u}K_\mu(u)e^{-vu}\sqrt{u}du}{\cos^2(\pi\mu) K_\mu^2(u)+(\pi I_\mu(u)+\sin(\pi \mu) K_\mu(u))^2},
\end{eqnarray*}
which implies
\begin{eqnarray*}
  \lim_{x\to \infty} \frac{\lambda}{x^{\mu-1/2}}w_{\lambda}(v) = w_\infty(v)\/,\quad v>0\/.
\end{eqnarray*}
Using (\ref{w1:estimate:abs}) and (\ref{w2:estimates}) from Appendix we get that
\begin{eqnarray}
   \label{wlambda:estimate}
   \frac{\lambda}{x^{\mu-1/2}}|w_{\lambda}(v)|\leq \frac{\lambda}{x^{\mu-1/2}}|w_{1,\,\lambda}(v)|+\frac{\lambda}{x^{\mu-1/2}}|w_{2,\,\lambda}(v)|\leq c_4 e^{-v\theta_\mu}+c_4 \frac{1}{(v+1)^{\mu+3/2}}
\end{eqnarray}
for some positive $\theta_\mu$. Next, we take advantage of the  following formula of the Laplace transform of $ w_\lambda(v)$ (see Lemma 3.1 in \cite{BR:2006})  
\begin{eqnarray*}
   \frac{\lambda}{x^{\mu-1/2}} \int_0^\infty e^{-r v}w_\lambda(v)dv = \frac{re^{\lambda r}x^{1/2} K_\mu(xr)}{ K_\mu(r)}-(r-(\mu^2-1/4)\frac{\lambda}{2x})\/,\quad r> 0.
\end{eqnarray*}
 Thus, taking the limit as $x$ tends to $\infty$ in the above relation
and applying the dominated convergence theorem we get 
\begin{eqnarray}
    \label{winfty:laplace}
   \int_0^\infty e^{-rv}w_\infty(v)dv = \sqrt{\frac{\pi r}{2}}\frac{e^{-r}}{K_\mu(r)}-r+\frac{\mu^2-1/4}{2}
\end{eqnarray}
for every $r> 0$.
Now let $c>0$. Observe that $\lim_{x/t\to c,x\to \infty}\frac\kappa{2t} = cv$. Using (\ref{wlambda:estimate}) and the dominated convergence theorem   we get
\begin{eqnarray*}
   \lim_{x/t\to c,\, x\to \infty} \frac{\lambda}{x^{\mu-1/2}}\int_0^\infty e^{-\kappa/2t}w_\lambda(v)dv &=& \int_0^\infty e^{-cv}w_\infty(v)dv\/.
\end{eqnarray*}
 We have
\begin{eqnarray*}
   \frac{\qmu\sqrt{2\pi t}}{\lambda e^{-\lambda^2/2t}} = \frac{x^{\mu-1/2}}{t}-(\mu^2-1/4)\frac{x^{\mu-1/2}}{2x}+\int_0^\infty e^{-\kappa/2t}w_\lambda(v)dv\/.
\end{eqnarray*}
Multiplying both sides by $\dfrac{\lambda}{x^{\mu-1/2}}$, taking limit as $x/t\rightarrow c$ and $x\rightarrow \infty$ and using (\ref{winfty:laplace}) we get
\begin{eqnarray*}   
   \lim_{x/t\to c,\,x\to\infty} \dfrac{1}{x^{\mu-1/2}}\frac{\qmu\sqrt{2\pi t}}{ e^{-\lambda^2/2t}} &=& c-\frac{\mu^2-1/4}{2}+\int_0^\infty e^{-cv}w_\infty(v)dv = \sqrt{\frac{\pi c}{2}} \frac{e^{-c}}{K_\mu(c)}>0\/.
\end{eqnarray*}
\end{proof}
}

\begin{lem}
\label{qt:estimate:inside}
For every $C>0$ there is a constant $c_1>0$ depending on $C$ and {$\mu>0$} such that   
\begin{eqnarray*}
 \frac1{c_1}   \lambda \frac{ e^{-\lambda^2/2t }}{t^{3/2}}\,x^{\mu-1/2}\le \qmu\le c_1   \lambda \frac{ e^{-\lambda^2/2t }}{t^{3/2}}\,x^{\mu-1/2}\/,
\end{eqnarray*}
whenever $x<Ct$, $x>1$. 
\end{lem}
\begin{proof}
By Lemma~\ref{qt:estimate:zero} it is enough to consider $\mu>1/2$.
Let $0<C^\prime<C$. The fact that the limit given in Proposition~\ref{Prop:asymp:inside} exists and is strictly positive implies that for every $c\in[C^\prime,C]$ there exist $\varepsilon_c>0$, $D_c>1$ and $x_c>2$ such that
\begin{eqnarray*}
   \frac{1}{D_c}\leq\dfrac{1}{x^{\mu-1/2}}\frac{\qmu\sqrt{2\pi t}}{ e^{-\lambda^2/2t}}\leq D_c
\end{eqnarray*}
for every $(x/t,x)\in(c-\varepsilon_c,c+\varepsilon_c)\times (x_c,\infty)$. The family 
\begin{eqnarray*}
   \left\{(c-\varepsilon_c,c+\varepsilon_c)\right\}_{c\in[C^\prime,C]}
\end{eqnarray*}
is an open cover of the compact set $[C^\prime,C]$. Consequently, there exists a finite subcover $\{(c_k-\varepsilon_{c_k},c_k+\varepsilon_{c_k})\}_{k=1,\ldots,m}$. Setting $C^*=\max\{x_{c_k}:k=1,\ldots,m\}$ and $D = \max\{D_{c_k},k=1,\ldots,m\}$ we get 
\begin{eqnarray}\label{middle_estimate}
   \frac{1}{D}\leq\dfrac{1}{x^{\mu-1/2}}\frac{\qmu\sqrt{2\pi t}}{ e^{-\lambda^2/2t}}\leq D
\end{eqnarray}
for every $C^\prime\leq x/t\leq C$ and $x>C^*$, which proves the lemma for this range of $x$ and $t$. Choosing small enough the constant $C^\prime$ depending on $\mu$ we infer that using Lemma \ref{qt:estimate:zero} we complete the proof in the case $ x/t\leq C$ and $x>C^*$.

   The estimates for $x\leq C^*$ and $C^\prime\leq x/t\leq C$ can be deduced from the absolute continuity property for Bessel processes with different indices. Indeed, from (\ref{Bessel:AC}) we have
\begin{eqnarray*}
x^{\mu-1/2} e^{-c_\mu t}q_x^{(-1/2)}(t)\leq \qmu\leq x^{\mu-1/2}q_x^{(-1/2)}(t)\/,
\end{eqnarray*}
where $c_\mu=\frac{\mu^2-1/4}2>0$.
Moreover, we have
\begin{eqnarray*}
  q_x^{(-1/2)}(t) = \lambda \frac{e^{-\lambda^2/2t}}{\sqrt{2\pi t^3}}
\end{eqnarray*}
and $\exp(-c_\mu t) \geq  \exp(-c_\mu C^*/C^\prime)$. This ends the proof.

\end{proof}

In the next lemma we show a result which provides a satisfactory estimate when $t$ is large relative to $x$. This is done under some additional assumption on $\mu$.
  
    \begin{lem}\label{qt:estimate:infty1}
     {Suppose that $\mu-1/2\in \N$. We have the following expansion
     $$  \qmu = \frac{(x^{2\mu}-1)}{{\Gamma(\mu)2^{\mu}}} e^{-\lambda^2/2t} \frac{1}{t^{\mu+1}}(1+ E(x, t)).$$
    There is a  constant $c>0$  such that  for $t>0$,
$$|E(x, t)|\le c \frac xt.$$}


 \end{lem}
 
 \begin{proof} We use the following result proved in Lemma 4.4 of \cite{BR:2006}.
 Let  $l=\mu-1/2$.
 Then
\begin{eqnarray*}
\lim_{t\to\infty} t^{l+1} \int_0^\infty w_\lambda(v)
\left(e^{-\kappa/2t}-\sum_{0\leq j\leq l} (-1)^j\frac{1}{j!}
 \left(\frac\kappa {2t}\right)^j\right)dv
 \end{eqnarray*}
  \begin{eqnarray} \label{cm}
 = \frac{(-1)^{l+1}}{2^{l+1} (l+1)!}\int_0^\infty\kappa^{l+1} w_\lambda(v)\ dv=C_{l+1}(x)>0\/.
 \end{eqnarray}
        Let   
      \begin{eqnarray*}
 H(\lambda,t)&=&\int_0^\infty w_\lambda(v)
\left(e^{-\kappa/2t}-\sum_{0\leq j\leq l+1} (-1)^j\frac{1}{j!}
 \left(\frac\kappa {2t}\right)^j\right)dv.
 \end{eqnarray*}
 Using (\ref{cm}) we may write    
      \begin{eqnarray*} 
 \lambda \frac{e^{-\lambda^2/2t} }{ \sqrt{2\pi t} } H(\lambda,t)&=&  \lambda \frac{e^{-\lambda^2/2t} }{ \sqrt{2\pi t} }\int_0^\infty w_\lambda(v)
\left(e^{-\kappa/2t}-\sum_{0\leq j\leq l+1} (-1)^j\frac{1}{ j!}
 \left(\frac\kappa {2t}\right)^j\right)dv\\&=&\lambda \frac{e^{-\lambda^2/2t} }{ \sqrt{2\pi t} }\int_0^\infty w_\lambda(v)
\left(e^{-\kappa/2t}-\sum_{0\leq j\leq l} (-1)^j\frac{1}{j!}
 \left(\frac\kappa {2t}\right)^j\right)dv \\
 &&- \lambda \frac{e^{-\lambda^2/2t} }{ \sqrt{2\pi t} }C_{l+1}t^{-l-1}\\
 &=& \qmu- \lambda \frac{e^{-\lambda^2/2t} }{ \sqrt{2\pi t} }C_{l+1}t^{-l-1},
 \end{eqnarray*}
 where  we applied (\ref{rep3}) in the last step.
 Observe that $$ \kappa^{l+2}\le c( \lambda^{l+2}v^{l+2}+v^{2l+4}),$$
 for some  constant $c$. Next,
  $$\left|e^{-\kappa/2t} - \sum_{0\leq j\leq l+1} (-1)^j\frac{1}{ j!}
      (\frac\kappa {2t})^j \right| \le  \left(\frac\kappa {2t}\right)^{l+2}\le c\left( \frac{\lambda^{l+2}v^{l+2}+v^{2l+4}}{t^{l+2}}\right),$$
      which together with the estimate (see Lemma \ref{w1:estimate:lemma} in Appendix)
$$ |w_{\lambda}(v)|=|w_{1,\lambda}(v)|\le  c x^{\mu-3/2} e^{-\theta_\mu v},$$
 leads to the following bound for       
      $H(\lambda,t)$:

 \begin{eqnarray}\label{H:estimate}
     | H(\lambda,t)| \le 
      \int_0^\infty \left(\frac\kappa {2t}\right)^{l+2}|w_\lambda(v)|dv\le c x^{\mu-3/2} \frac{(\lambda^{l+2}+1)}{t^{l+2}}\approx 
       \frac{x^{2\mu}}{t^{l+2}}.
      \end{eqnarray}  
   To complete the proof we need   
       to find the constant $C_{l+1}$. Let $T_0^{(-\mu)}$ denote the hitting time of  $0$ if we start the process from $x$. Due to the strong Markov property and the scaling property  we obtain the following equality of the distributions:
 $$T_0^{(-\mu)}{\stackrel{d}{=}}\frac1 {x^2}T_0^{(-\mu)} + T_1^{(-\mu)},$$
 where $T_0^{(-\mu)}$ and $ T_1^{(-\mu)}$ are independent. It follows that  
 
 $${P_x^{(-\mu)}}(T_0^{(-\mu)}>t)\cong  {P_x^{(-\mu)}}(T_0^{(-\mu)}>x^2t) +{P_x^{(-\mu)}}(T_1^{(-\mu)}>t),\ t\to \infty.$$
 {Note that by the result of Getoor and Sharpe \cite{GetoorSharpe:1979}} we know that  $t^{\mu}{P_x^{(-\mu)}}(T_0^{(-\mu)}>t)\cong\frac {x^{2\mu}}{\Gamma(\mu+1)2^{\mu}}$, which implies that 
 $$t^{\mu}{P_x^{(-\mu)}}(T_1^{(-\mu)}>t)\cong\frac {x^{2\mu}-1}{\Gamma(\mu+1)2^{\mu}}.$$
 Fom (\ref{H:estimate}) and 
  \begin{eqnarray*} 
\qmu &=& 
 \lambda \frac{e^{-\lambda^2/2t} }{ \sqrt{2\pi t} }C_{l+1}t^{-l-1}+\lambda \frac{e^{-\lambda^2/2t} }{ \sqrt{2\pi t} } H(\lambda,t)
 \end{eqnarray*}
 we infer
 
 $$t^{\mu}{P_x^{(-\mu)}}(T_1^{(-\mu)}>t)\cong \frac{\lambda }{\mu \sqrt{2\pi } }C_{l+1}.$$
  This in turn shows that 
 %
 %
  $C_{l+1}(x)=\frac {\sqrt{2\pi}\mu(x^{2\mu}-1)}{\lambda\Gamma(\mu+1)2^{\mu}}=\frac {\sqrt{2\pi}(x^{2\mu}-1)}{\lambda\Gamma(\mu)2^{\mu}}$
      and completes the proof.

    \end{proof}
    
    \begin{rem} Let $l= \mu-1/2\in\N$ and  $k\in \N$. Since all moments of $\kappa$  with respect to $w_\lambda(v)dv$  exist we can write
     $$ | {\qmu}- \lambda \frac{e^{-\lambda^2/2t} }{ \sqrt{2\pi t} }\sum_{i=1}^k \frac{C_{l+i}}{t^{l+i}}|\le c\lambda \frac{e^{-\lambda^2/2t} }{ \sqrt{2\pi t} }\frac{x^{2\mu+k-1}}{t^{l+k+1}}, $$
 for some constant $c$, { depending only on $\mu$ and $k$},   where the constants $C_{l+i}= C_{l+i}(x)$ can be found by similar considerations as above.
     \end{rem}
\begin{lem} Let $\mu-1/2\notin \N$ and let $l=[\mu+1/2]$.  There are constants $c_1, c_2, c_3$ depending only on $\mu$ such that
\label{qt:estimate:infty2}
$$c_2\frac{\lambda x^{2\mu-1}} { t^{\mu+1}} e^{-\lambda^2/2t}\left( 1-c_3\left(\frac{x} { t}\right)^{l-\mu+1/2}\right) \le q^{(-\mu)}_x(t)\le c_1 \frac{\lambda x^{2\mu-1}} { t^{\mu+1}} e^{-\lambda^2/2t},$$
 for $t>x>1$.
Note that $l-\mu+1/2>0$.
\end{lem}

\begin{proof}
Applying (\ref{rep3}) we have 
\begin{eqnarray*}
     \qmu& =& \lambda \frac{e^{-\lambda^2/2t} }{ \sqrt{2\pi t} }
      \int_0^\infty \left(e^{-\kappa/2t} - \sum_{0\leq j\leq l} (-1)^j\frac{1}{ j!}
      \(\frac\kappa {2t}\right)^j \)w_\lambda(v)dv\\
    & =& \lambda \frac{e^{-\lambda^2/2t} }{ \sqrt{2\pi t} }
      \int_0^\infty \left(e^{-\kappa/2t} - \sum_{0\leq j\leq l} (-1)^j\frac{1}{ j!}
      \(\frac\kappa {2t}\right)^j \)w_{\lambda,1}(v)dv\\
      &+& \lambda \frac{e^{-\lambda^2/2t} }{ \sqrt{2\pi t} }
      \int_0^\infty \left(e^{-\kappa/2t} - \sum_{0\leq j\leq l} (-1)^j\frac{1}{ j!}
      \(\frac\kappa {2t}\right)^j \)w_{\lambda,2}(v)dv \\
      &=& q^{(-\mu)}_{x,1}(t)+ q^{(-\mu)}_{x,2}(t)
      .
      \end{eqnarray*}
  The upper estimate of $q^{(-\mu)}_{x,1}(t)$ is obtained using almost the same arguments  as in the proof of (\ref{H:estimate}) in Lemma \ref{qt:estimate:infty1}. The resulting bound is of the following form:
  
          \begin{eqnarray}
 |q^{(-\mu)}_{x,1}(t)|\le c x^{2\mu-1}\lambda \frac{e^{-\lambda^2/2t} }{  t^{\mu+1} }\left(\frac{x}{t}\right)^{l-\mu+1/2}.\label{q1:estimate}
 \end{eqnarray}
 Next, we deal with  $q^{(-\mu)}_{x,2}(t)$. 
  Observing that 
  
 $$ e^{-\kappa/2t} - \sum_{0\leq j\leq l} (-1)^j\frac{1}{ j!}
      \(\frac\kappa {2t}\right)^j \approx(-1)^{l+1} \frac {\kappa^{l+1}}{t^l(\kappa +t)}$$
  and using Lemma \ref{w2:estimate:lemma} from Appendix, where the estimate of $w_{\lambda,2}(v)$ is provided,  we have

\begin{eqnarray*}
     q^{(-\mu)}_{x,2}(t)
&\approx&\lambda \frac{e^{-\lambda^2/2t} }{ t^{l+1/2} }
      \int_0^\infty \frac {x^{2\mu-1}}{(v+1)^{\mu+3/2}(v+x)^{\mu+1/2}}\frac {\kappa^{l+1}}{(\kappa +t)}dv.
      \end{eqnarray*}
  We need to effectively estimate the integral   
  
   $$ J(t,x)=\int_0^\infty \frac {x^{2\mu-1}}{(v+1)^{\mu+3/2}(v+x)^{\mu+1/2}}\frac {\kappa^{l+1}}{(\kappa +t)}dv.$$
       Using the folowing  change of variables
$\frac {\kappa}t= s$ we obtain 
$\frac {\kappa+v^2}{v t}  dv=  ds,$ which yields

\begin{eqnarray}\label{change}\frac { dv} v\le \frac { ds} s\le 2\frac { dv} v. \end{eqnarray}
   
     Assume that $x>2$.  
 Thus,  $\kappa=2\lambda v+v^2\approx v(v+x)$ and 
      
   $$ J(t,x)=  \int_0^\infty \frac {x^{2\mu-1}(v+x)^{l-\mu+1/2}v^{l+1}}{(v+1)^{\mu+3/2}}\frac {1}{(\kappa +t)}dv=J_1(t,x)+J_2(t,x),$$
    where  
 $$ J_1(t,x)=  \int_0^1\frac {x^{2\mu-1}(v+x)^{l-\mu+1/2}v^{l+1}}{(v+1)^{\mu+3/2}}\frac {1}{(\kappa +t)}dv\approx
 x^{\mu+l-1/2}\int_0^1\frac {v^{l+1}}{(\kappa +t)}dv$$
and        
 $$ J_2(t,x)=  \int_1^\infty\frac {x^{2\mu-1}(v+x)^{l-\mu+1/2}v^{l+1}}{(v+1)^{\mu+3/2}}\frac {1}{(\kappa +t)}dv\approx
 x^{2\mu-1}\int_1^\infty\frac {\kappa^{l-\mu+1/2}}{v(\kappa +t)}dv.$$
  Applying the above change of variables and (\ref{change}) we obtain

$$J_2(t,x)\approx x^{2\mu-1} \int_1^\infty\frac {\kappa^{l-\mu+1/2}}{v(\kappa +t)}dv\approx x^{2\mu-1}\int_{(1+x)/t}^\infty\frac {(st)^{l-\mu+1/2}}{s(s+1)t}ds=x^{2\mu-1}\frac {t^l} {t^{\mu+1/2}}\int_{(1+x)/t}^\infty \frac{s^{l-\mu-1/2}}{s+1}ds.$$
Observing  that  $\int_{0}^\infty \frac{s^{l-\mu-1/2}}{s+1}ds<\infty$ we arrive at 

$$J_2(t,x)\approx x^{2\mu-1}  \frac {t^l} {t^{\mu+1/2}},\quad t>x.$$
The first integral $ J_1(t,x)$ we trivially estimate 
$$ J_1(t,x) \le c
 x^{2\mu-l}\left(\frac {x}{t}\right)^{l+1/2-\mu}t^{l-1/2-\mu}\le c J_2(t,x).$$
This yields the following estimate 
\begin{eqnarray}\label{integral1}
J(t,x)\approx x^{2\mu-1}  \frac {t^l} {t^{\mu+1/2}},\quad  t>x.\end{eqnarray}
   Next, assume that $x\le 2$.
 Thus,  $\kappa\approx v^2,\ v\ge 1$, and

   $$ J(t,x)\approx \int_0^\infty \frac {1}{(v+1)^{2\mu+2}}\frac {(\lambda v)^{l+1}+ v^{2(l+1)} }{(\kappa +t)}dv= J_3(t,x)+ J_4(t,x),$$
   where $$  J_4(t,x)= \int_1^\infty \frac {1}{(v+1)^{2\mu+2}}\frac {(\lambda v)^{l+1}+ v^{2(l+1)} }{(\kappa +t)}dv\approx \int_1^\infty \frac {v^{2(l-\mu)} }{(v^2 +t)}dv \approx \frac {t^l} {t^{\mu+1/2}}.$$
   Clearly 
  $$  J_3(t,x)= \int_0^1 \frac {1}{(v+1)^{2\mu+2}}\frac {(\lambda v)^{l+1}+ v^{2(l+1)} }{(\kappa +t)}dv\le c\frac 1t\le c  J_4(t,x)$$
  if $t>1$.
  Obviously this implies that (\ref{integral1}) holds in the case $x\le2$. 
  
  Using this estimate we finally obtain that 
\begin{eqnarray}
       q^{(-\mu)}_{x,2}(t)
&\approx&\lambda \frac{e^{-\lambda^2/2t} }{ t^{l+1/2} }
      J(t,x)\nonumber\\
     &\approx& \lambda x^{2\mu-1} \frac{e^{-\lambda^2/2t}} { t^{\mu+1}}, \quad t>x>1.\label{q2:estimate} \end{eqnarray}
A combination of (\ref{q1:estimate}) and (\ref{q2:estimate}) ends the proof.
\end{proof}

\begin{lem} Let $\mu=0$. 
\label{qt:zero:estimate:infty}
  For $t>2x$,
  
  \label{qt:estimate:large}
$$ q_x^{(0)}(t) \approx \frac{\lambda} x\frac{e^{-\lambda^2/2t} }{ t }
  \frac {1+\log x}{(1+\log \frac tx)(1+\log t)}. $$

\end{lem}

\begin{proof}
 Recalling  the representation (\ref{rep1}) for $\qmu$ and observing that $e^{-\kappa/2t} - 1\approx \frac {-\kappa}{(\kappa +t)}$ we have (note that $w_\lambda(v)\le 0$)  
\begin{eqnarray*}q_x^{(0)}(t) &=& \lambda \frac{e^{-\lambda^2/2t} }{ \sqrt{2\pi t} }
   \(x^{-1/2}/2t +
    \int_0^\infty \(e^{-\kappa/2t} - 1\) w_\lambda(v)dv \)\\
  &\approx& \lambda \frac{e^{-\lambda^2/2t} }{ \sqrt{2\pi t} }
   \(x^{-1/2}/2t +
    \int_0^\infty \frac {\kappa}{(\kappa +t)}(-w_\lambda(v))dv \)  
    .\end{eqnarray*}
Hence, it is enough to estimate the integral 
 $$J(t,x)=\int_0^\infty \frac {\kappa}{(\kappa +t)}(-w_\lambda(v))dv.$$

We start with the case $x>2$. Due to Lemma \ref{w:muzero:estimate:lemma} (see Appendix) we have
$$ -w_\lambda(v) 
\approx\left\{
\begin{array}{ll}
 \frac{1}{x^{3/2}}, & \hbox{$v<3/2$,} \\
    \frac{1}{x^{3/2}v^{3/2}\log v}, & \hbox{$3/2 \leq v<x$,}\\
   \frac{\log x}{xv^2\log^2v}, &\hbox{$v\ge x>2$.}\end{array}
\right.$$
We write
  
\begin{eqnarray*}
      J(t,x)&=& \(\int_0^2 + \int_2^x + \int_x^\infty\) \frac {\kappa}{(\kappa +t)}(-w_{\lambda}(v))dv
      \\
      &=& J_1(t,x)+J_2(t,x)+ J_3(t,x).
      \end{eqnarray*}
  Note that 
$$ \frac {\kappa}{(\kappa +t)}
\approx\left\{
\begin{array}{ll}
 \frac {vx}{(vx +t)}, & \hbox{$v<x$,} \\
    \frac {v^2}{(v^2 +t)}, & \hbox{$ v\ge x$,}\end{array}
\right.$$
  which implies
   $$ J_1(t,x)\approx \int_0^{3/2} \frac{1}{x^{3/2}}\frac {\kappa}{(\kappa +t)}dv\approx \int_0^{3/2} \frac{1}{x^{3/2}}\frac {vx}{(vx +t)}dv\le \frac{3}{2x^{1/2}t},  $$
   
      $$ J_2(t,x)\approx \int_{3/2}^x \frac{1}{x^{3/2}v^{3/2}\log v}\frac {\kappa}{(\kappa +t)}dv\approx \int_{3/2}^x \frac{1}{x^{3/2}v^{3/2}\log v}\frac {vx}{(vx +t)}dv,$$
   $$  J_3(t,x)(t,x)\approx \int_x^\infty \frac{\log x}{xv^2\log^2v}\frac {\kappa}{(\kappa +t)}dv\approx \int_x^\infty \frac{\log x}{xv^2\log^2v}\frac {v^2}{(v^2 +t)}dv.$$
Assume that $2x<t<x^2$.  First, we   deal with 
  \begin{eqnarray*} J_2(t,x)&\approx& \frac{1}{x^{1/2}}\int_{3/2}^{t/x}\frac{1} {v^{1/2}\log v}\frac {1}{(vx +t)}dv+\frac{1}{x^{1/2}}\int_{t/x}^x\frac{1} {v^{1/2}\log v}\frac {1}{(vx +t)}dv\\
  &\approx& \frac{1}{x^{1/2}t}\int_{3/2}^{t/x}\frac{1} {v^{1/2}\log v}dv+\frac{1}{x^{3/2}}\int_{t/x}^x\frac{1} {v^{3/2}\log v}dv\\&\approx& \frac{1}{x^{1/2}t}\int_{3/2}^{t/x}\frac{1} {v^{1/2}\log v}dv+\frac{1}{x^{3/2}}\int_{t/x}^x\frac{1} {v^{3/2}\log v}dv.\end{eqnarray*}
  We have 
  $$
  \int_{3/2}^{t/x}\frac{1} {v^{1/2}\log v}dv\approx \sqrt{t/x}\frac 1{\log \frac tx}
  $$
 and
 $$
  \int_{t/x}^x\frac{1} {v^{3/2}\log v}dv\le 2 \sqrt{x/t}\frac 1{\log \frac tx},
  $$
which shows that 
   $$ J_2(t,x)\approx  \frac 1{x\sqrt{t}}\frac 1{\log \frac tx}.$$
   Next,
   
   $$  J_3(t,x)(t,x)\approx \int_x^\infty \frac{\log x}{x\log^2v}\frac {1}{(v^2 +t)}dv\approx \frac{\log x}x \int_x^\infty \frac{1}{v^2\log^2v}dv\approx \frac  1{ x^2\log x}\/. $$
   Combining all the estimates we see that for  $2x<t<x^2$ we have 
   
   $$ J(t,x)\approx J_2(t,x)\approx  \frac 1{x\sqrt{t}}\frac 1{\log \frac tx}.$$
   
   Next, we assume that  $t>x^2$. 
      \begin{eqnarray*}  J_3(t,x)(t,x)&\approx& \int_x^{\sqrt{t}} \frac{\log x}{x\log^2v}\frac {1}{(v^2 +t)}dv+ \int_{\sqrt{t}}^\infty \frac{\log x}{x\log^2v}\frac {1}{(v^2 +t)}dv\\
 &\approx&    \frac{\log x}{xt}\int_x^{\sqrt{t}} \frac{1}{\log^2v}dv+ \frac{\log x}{x}\int_{\sqrt{t}}^\infty \frac{1}{v^2\log^2v}dv.  \end{eqnarray*}
 Observe that
 
$$ \int_x^{\sqrt{t}} \frac{1}{\log^2v}dv\le \int_2^{\sqrt{t}} \frac{1}{\log^2v}dv\approx \frac{\sqrt{t}}{\log^2t}$$ and

$$\int_{\sqrt{t}}^\infty \frac{1}{v^2\log^2v}dv \approx \frac{1}{\sqrt{t}\log^2t}.$$
As a consequence we obtain 

 $$ J_3(t,x)(t,x)\approx  \frac 1{x\sqrt{t}}\frac {\log x}{\log^2 t}.$$
 Next,
 \begin{eqnarray*} J_2(t,x)&\approx& \frac{1}{x^{1/2}}\int_{3/2}^x\frac{1} {v^{1/2}\log v}\frac {1}{(vx +t)}dv\\
  &\approx& \frac{1}{x^{1/2}t}\int_{3/2}^{x}\frac{1} {v^{1/2}\log v}dv\\&\approx& \frac{1}{t\log x}\le C\frac 1{x\sqrt{t}}\frac {\log x}{\log^2 t},\ x^2\le t.\end{eqnarray*}
  Recall that 
   $$ J_1(t,x)\le \frac{3}{2x^{1/2}t}.  $$
  Hence, in this case, it is easily seen that the integral $ J_3(t,x)$ dominates and 
   $$ J(t,x)\approx  J_3(t,x)(t,x)\approx  \frac 1{x\sqrt{t}}\frac {\log x}{\log^2 t},\ x^2\le t.$$
  Summarizing all the estimates obtained for $ J(t,x)$ we have that for $4<2x<t<x^2$, 
    \begin{eqnarray*}
      q_x^{(0)}(t) &=& \lambda \frac{e^{-\lambda^2/2t} }{ \sqrt{2\pi t} }
   \(x^{-1/2}/2t +
    \int_0^\infty \(e^{-\kappa/2t} - 1\) w_\lambda(v)dv \)\\
   &\approx& \lambda \frac{e^{-\lambda^2/2t} }{ \sqrt{2\pi t} }
   \(x^{-1/2}/2t + \frac 1{x\sqrt{t}}\frac 1{\log \frac tx}\)\\
   &\approx& \frac{e^{-\lambda^2/2t} }{ t }
  \frac 1{\log \frac tx}\\
  &\approx&
  \frac{\lambda} x\frac{e^{-\lambda^2/2t} }{ t }
  \frac {1+\log x}{(1+\log \frac tx)(1+\log t)},
    \end{eqnarray*}
   while  for $t>x^2$, 
    \begin{eqnarray*}
  q_x^{(0)}(t) &=& \lambda \frac{e^{-\lambda^2/2t} }{ \sqrt{2\pi t} }
   \(x^{-1/2}/2t +
    \int_0^\infty \(e^{-\kappa/2t} - 1\) w_\lambda(v)dv \)\\
   &\approx& \lambda \frac{e^{-\lambda^2/2t} }{ \sqrt{2\pi t} }
   \(x^{-1/2}/2t + \frac 1{x\sqrt{t}}\frac {\log x}{\log^2 t}\)\\
  &\approx& \frac{e^{-\lambda^2/2t} }{ t }
  \frac {\log x}{\log^2 t} \\
  &\approx&   \frac{\lambda} x\frac{e^{-\lambda^2/2t} }{ t }
  \frac {1+\log x}{(1+\log \frac tx)(1+\log t)}.
    \end{eqnarray*}
  This completes the proof in the case $x>2$. 
  
  Finally, assume that $1<x\le 2\le t$. By  Lemma \ref{w:muzero:estimate:lemma} (see Appendix) we have
  
  $$ -w_\lambda(v) 
\approx\left\{
\begin{array}{ll}
 1, & \hbox{$v\le 2$,} \\
    \frac{1}{v^2(\log^2 v+1)}, & \hbox{$v>2 $.}\\
   \end{array}
\right.$$
  Thus, using the fact that for such range of $x$ we have
  \begin{eqnarray*}
     1-e^{-\kappa/2t} \approx \frac{\kappa}{\kappa+t}\approx \frac{v^2}{v^2+t}
  \end{eqnarray*}
  we obtain
  \begin{eqnarray*}
    q_x^{(0)}(t) &\approx& \lambda\frac{e^{-\lambda^2/2t}}{\sqrt{t}}\left(\frac{1}{x^{1/2}t}+\int_0^2\frac{v^2}{v^2+t}dv+\int_2^\infty \frac{dv}{(v^2+t)(\log^2v+1)} \right)\/.
  \end{eqnarray*}
  Obviously, we have
  \begin{eqnarray*}
     \int_0^2\frac{v^2}{v^2+t}dv \approx \frac{1}{t}\
  \end{eqnarray*}
and
  \begin{eqnarray*}
     \int_2^\infty \frac{dv}{(v^2+t)(\log^2v+1)}  = \frac{1}{\sqrt{t}}\int_{\frac{2}{\sqrt{t}}}^\infty \frac{1}{(s^2+1)}\frac{ds}{(\log(s)+\log{\sqrt{t}})^2+1}\approx \frac{1}{\sqrt{t}\log^2 t}\/.
  \end{eqnarray*}
Hence, for $1<x\le 2\le t$,  we have
\begin{eqnarray*}
 q_x^{(0)}(t)&\approx& \lambda\frac{e^{-\lambda/4t}}{t\log^2 t} \approx \lambda\frac{e^{-\lambda/4t}}{t}\frac{1+\log x}{\log^2 t} \\ &\approx&\frac{\lambda} x\frac{e^{-\lambda^2/2t} }{ t }
  \frac {1+\log x}{(1+\log \frac tx)(1+\log t)}.
\end{eqnarray*}
The proof is completed.
\end{proof}


     \section{Applications}
    
     \subsection{Survival probabilities of killed Bessel process}
     In this subsection  we introduce uniform estimates for the survival probabilities of Bessel process killed when exiting the half-line $(1,\infty)$. 
The theorem below is formulated for  processes with non-negative indices, however due to (\ref{BesselTime:indices}) we can easily derive the corresponding result for strictly negative indices.
     \begin{thm}
        Let $\mu>0$. Then, for every $t\geq 0$ and $x>1$, we have
        \begin{eqnarray*}
           \textbf{P}_x^{(\mu)}(t<T^{(\mu)}_1<\infty) \approx \frac{x-1}{\sqrt{x\wedge t}+x-1}{\frac{1}{t^{\mu}+x^{2\mu}}}\/.
        \end{eqnarray*}
       Moreover, for every $t\geq 0$ and $x>1$, we have
        \begin{eqnarray*}
         \textbf{P}_x^{(0)}(T^{(0)}_1>t) &\approx&  1\wedge\frac{\log x}{\log(1+ t^{1/2})}\/.\end{eqnarray*}
     \end{thm}
     \begin{proof}
       Using (\ref{Bessel:hittingtime:estimates}) we get
       \begin{eqnarray*}
          \textbf{P}_x^{(\mu)}(t<T^{(\mu)}_1\lefteqn{<\infty) = \int_t^\infty q_x^{(\mu)}(s) ds}\\
          &\approx& \frac{(x-1)}{x^{2\mu}}\left(x^{\mu-1/2}\int_{t}^{t\vee x} s^{-3/2}e^{-(x-1)^2/2s}ds + x^{2\mu-1}\int_{t\vee x}^\infty s^{-\mu-1}e^{-(x-1)^2/2s}ds \right)\\
          &=& J_1(t,x)+J_2(t,x)\/.
       \end{eqnarray*}
       The integral $J_2(t,x)$ can be estimated as follows
       \begin{eqnarray*}
          J_2(t,x) &=& \frac{x-1}{x}\int_{t\vee x}^\infty s^{-\mu-1}e^{-(x-1)^2/2s}ds = \frac{(x-1)^{1-2\mu}}{x}\int_0^{\frac{(x-1)^2}{2(t\vee x)}}u^{\mu-1}e^{-u}du\\
          &\approx& \frac{(x-1)^{1-2\mu}}{x}\left(\frac{(x-1)^{2\mu}}{(t\vee x)^{\mu}}\wedge 1\right) = \frac{x-1}{x}\left(\frac{1}{(t\vee x)^{\mu}}\wedge\frac{1}{(x-1)^{2\mu}}\right)\\
          &\approx& \frac{x-1}{x}\frac{1}{(t\vee x)^{\mu}+(x-1)^{2\mu}}\/.
       \end{eqnarray*}
       Observe also that for $t\geq x$ we have
       \begin{eqnarray*}
         \frac{1}{(t\vee x)^{\mu}+(x-1)^{2\mu}} &=& \frac{1}{t^{\mu}+(x-1)^{2\mu}} \approx \frac{1}{t^\mu+x^{2\mu}}\/,\\
         \frac{x-1}{x} &\approx& \frac{x-1}{\sqrt{x}-1+x} = \frac{x-1}{\sqrt{x\wedge t}+x-1}\/.
       \end{eqnarray*}
       The fact that for $t\geq x$ the integral $J_1(t,x)$ vanishes together with the above-given estimates of $J_2(t,x)$ end the proof in that case. If $t<x$, then
       \begin{eqnarray*}
          J_2(t,x) &\approx& \frac{x-1}{x}\frac{1}{x^{\mu}+(x-1)^{2\mu}} \approx \frac{x-1}{x}\frac{1}{t^{\mu}+x^{2\mu}}\/.
       \end{eqnarray*}
       Substituting $u=(x-1)^2/(2t)$ we can rewrite $J_1(t,x)$ in the following way
       \begin{eqnarray*}
        J_1(t,x) &=& \frac{x-1}{x^{\mu+1/2}}\int_{t}^{x} s^{-3/2}e^{-(x-1)^2/2s}ds
        = \frac{\sqrt{2}}{x^{\mu+1/2}}\int_{\frac{(x-1)^2}{2x}}^{\frac{(x-1)^2}{2t}}u^{-1/2}e^{-u}du\/,\quad t<x\/.
       \end{eqnarray*}
       For $x\geq 2$ we have
       \begin{eqnarray*}
          J_1(t,x) &\leq& \frac{1}{2^{\mu}}\int_{\frac{(x-1)^2}{2x}}^\infty u^{-1/2}e^{-u}du\approx \left(\frac{(x-1)^2}{2x}\right)^{1/2}\exp\left({-\frac{(x-1)^2}{2x}}\right)
       \end{eqnarray*}
       and it means that $J_1(t,x)$ is dominated by $J_2(t,x)$ in that region. Consequently, we obtain
       \begin{eqnarray*}
         \textbf{P}_x^{(\mu)}(t<T^{(\mu)}_1<\infty) &\approx& J_2(t,x)\approx \frac{x-1}{x}\frac{1}{t^{\mu}+x^{2\mu}}\approx \frac{x-1}{\sqrt{t}+x-1}\frac{1}{t^{\mu}+x^{2\mu}}
       \end{eqnarray*}
       whenever $t<x$ and $x\geq 2$. Moreover, for $1<x<2$, $t<x$ and $(x-1)>\sqrt{t}$ we get
       \begin{eqnarray*}
         \frac{1}{2^{\mu}}\int_{1/4}^{1/2} u^{-1/2}e^{-u}du\leq  J_1(t,x)\leq \int_{0}^\infty u^{-1/2}e^{-u}du\/.
       \end{eqnarray*}
       Thus, the integral $J_1(t,x)$ dominates $J_2(t,x)$ and we have
       \begin{eqnarray*}
          \textbf{P}_x^{(\mu)}(t<T^{(\mu)}_1<\infty) &\approx& J_1(t,x) \approx 1\approx \frac{x-1}{\sqrt{t}+x-1}\frac{1}{t^\mu+u^{2\mu}}\/.
       \end{eqnarray*}
       Finally, for $1<x<2$, $t<x$ and $(x-1)\leq\sqrt{t}$ we get
       \begin{eqnarray*}
        J_1(t,x)\approx \int_{\frac{(x-1)^2}{2x}}^{\frac{(x-1)^2}{2t}}u^{-1/2}du =\sqrt{2}(x-1)\left(\frac{1}{\sqrt{t}}-\frac{1}{\sqrt{x}}\right)\/,\quad J_2(t,x)\approx x-1\/.
       \end{eqnarray*}
       Thus we get $1-1/\sqrt{x}\leq (\sqrt{x}-1)/\sqrt{t}\leq \sqrt{2}/\sqrt{t}$ and consequently
       \begin{eqnarray*}
         \textbf{P}_x^{(\mu)}(t<T^{(\mu)}_1<\infty) &\approx& (x-1)\left(\frac{1}{\sqrt{t}}-\frac{1}{\sqrt{x}}+1\right) \approx \frac{x-1}{\sqrt{t}} \approx \frac{x-1}{\sqrt{t}+x-1}\frac{1}{t^\mu+x^{2\mu}}\/.
       \end{eqnarray*}
       
       Now we deal with the case $\mu=0$. We begin with the case of large time $t\geq 2$. We have to consider three cases. For $s\geq t\geq x^2$ we have
       \begin{eqnarray*}
         q_x^{(0)}(s)  \approx \frac \lambda {x s}\frac {1+\log x}{\log^2 s}\/.
       \end{eqnarray*}
       Consequently
       \begin{eqnarray}
         \label{DuzeT}
         \textbf{P}_x^{(0)}(T^{(0)}_1>t) &\approx& \frac{x-1}x \int_t^\infty \frac {1+\log x}{s\log^2 s}ds\approx
         \frac{(x-1)(1+\log x)}{x\log t}
         \approx \frac{\log x}{\log(1+ t^{1/2})}\/.
       \end{eqnarray}
       If $2\leq t\leq x^2$ and additionally $x\geq 2$, using the above estimate, we have 
       \begin{eqnarray*}
         1 \geq \textbf{P}_x^{(0)}(T^{(0)}_1>t) \ge \textbf{P}_x^{(0)}(T^{(0)}_1>x^2)  \approx \frac{\log x}{\log(1+ x)}\approx 1\/.
       \end{eqnarray*}
       Finally, for $2\leq t\leq x^2$ with $x<2$ we get $2<t<4$ and we can write
       \begin{eqnarray*}
         \textbf{P}_x^{(0)}(T^{(0)}_1>t) &=& \textbf{P}_x^{(0)}(t<T^{(0)}_1\le 10) +  \textbf{P}_x^{(0)}(T^{(0)}_1\ge 10)\approx \lambda\/.
        \end{eqnarray*}  
        To justify the last approximation observe that, using Lemma \ref{qt:estimate:zero}, we get
       \begin{eqnarray*}
         q_x^{(0)}(s)  \approx\lambda,\quad 2\le t\le s\le 10\/,
       \end{eqnarray*}
       and it gives
       \begin{eqnarray*}
         \textbf{P}_x^{(0)}(t<T^{(0)}_1\le 10) = \int_t^{10} q_x^{(0)}(s)ds \approx \lambda\/.
       \end{eqnarray*}
       Moreover, using (\ref{DuzeT}) we also get 
       \begin{eqnarray*}
         \textbf{P}_x^{(0)}(T^{(0)}_1\ge 10) \approx \lambda\/.
       \end{eqnarray*} 
       Combining all cases we obtain that
       \begin{eqnarray*}
         \textbf{P}_x^{(0)}(T^{(0)}_1>t) 
         &\approx& 1\wedge\frac{\log x}{\log(1+ t^{1/2})},\quad t\ge 2 \/.
       \end{eqnarray*}
       In the case of small times $t\le 2$ and $1<x<2$, by Lemma \ref{qt:estimate:zero}, we have  
       \begin{eqnarray*}
           q_x^{(0)}(s)  \approx{\lambda} \dfrac{e^{-\lambda^2/2s} }{ s^{3/2} }, \quad  t\le s\le 10\/.
       \end{eqnarray*}    
       and thus
       \begin{eqnarray*}
         \textbf{P}_x^{(0)}(t<T^{(0)}_1\le 10)\approx \int_t^{10} {\lambda} \dfrac{e^{-\lambda^2/2s} }{ s^{3/2} }ds\approx 1\wedge \frac\lambda {t^{1/2}}
       \end{eqnarray*}
       Observe also that  $\textbf{P}_x^{(0)}(T^{(0)}_1\ge 10)\approx \lambda$ by our previous estimates in the case of large times. Hence 
        \begin{eqnarray*}
          \textbf{P}_x^{(0)}(T^{(0)}_1>t) &\approx& 1\wedge \frac\lambda {t^{1/2}}.
        \end{eqnarray*}       
      Finally, for $t\le 2$ and $x>2$, using the Markov property, one can easily obtain that
       \begin{eqnarray*}
        \textbf{P}_x^{(0)}(T^{(0)}_1>t)\ge   \textbf{P}_2^{(0)}(T^{(0)}_1>2)\approx 1\/.
       \end{eqnarray*}  
      Again combining all the cases we easily obtain 
       \begin{eqnarray*}
        \textbf{P}_x^{(0)}(T^{(0)}_1>t) &\approx& 1\wedge\frac \lambda{t^{1/2}}\approx 1\wedge\frac{\log x}{\log(1+ t^{1/2})}, \quad t\le 2 .
       \end{eqnarray*}
      This ends the proof.     \end{proof}
     \subsection{Poisson kernel for hyperbolic Brownian motion with drift}
     Let us consider a half-space model of $n$-dimensional real hyperbolic space
     \begin{eqnarray*}
       \H^n = \{(y_1,\ldots,y_{n-1},y_n)\in\R^n: y_n>0\}
     \end{eqnarray*}
     with {Riemannian} metric
     \begin{eqnarray*}
        ds^2 = \frac{dy_1^2+\ldots+dy_{n-1}^2+dy_n^2}{y_n^2}\/.
     \end{eqnarray*}
     The hyperbolic distance $d_{\H^n}(y,z)$ is given by 
     \begin{eqnarray*}
        \cosh d_{\H^n}(y,z) = 1+\frac{|y-z|^2}{{2y_nz_n}}\/,\quad y,z\in\H^n\/.
     \end{eqnarray*}
     The Laplace-Beltrami operator associated with the metric is given by
     \begin{eqnarray*}
        \Delta_{\H^n} = y_n^2\sum_{i=1}^n \dfrac{\partial^2}{\partial y_i^2}-(n-2)y_n\dfrac{\partial}{\partial y_n}\/.
     \end{eqnarray*}
     For every $\mu> 0$, we also define the following operator 
     \begin{eqnarray*}
        \Delta_\mu = \Delta_{\H^n}-(2\mu-n+1)y_n\frac{\partial}{\partial y_n} = y_n^2\sum_{i=1}^n \dfrac{\partial^2}{\partial y_i^2}-(2\mu-1)y_n\dfrac{\partial}{\partial y_n}\/.
     \end{eqnarray*} 
     The hyperbolic Brownian motion (HBM) with drift is a diffusion $Y^{(\mu)}=\{Y_t^{(\mu)},t\geq 0\}$ on $\H^n$ with a generator $\frac12\,\Delta_{\mu}$. For $\mu=\frac{n-1}{2}$ we obtain the standard HBM on $\H^n$ (with $\frac12\, \Delta_{\H^n}$ as a generator).
     
     The structure of the process $Y^{(\mu)}$ starting from $(\tilde{y},y_n)\in\H^n$ can be described in terms of the geometric Brownian motion and integral functional $\Amu$ as follows. Let $\tilde{B}=\{\tilde{B}_t,t\geq 0\}$ be $(n-1)$-dimensional Brownian motion starting from $\tilde{y}\in\R^{n-1}$ independent from a geometric Brownian motion $X^{(-\mu)}$ starting from $y_n>0$. Then we have
     \begin{eqnarray}
        \label{HBM:structure}
        Y^{(\mu)}_t \stackrel{d}{=} (\tilde{B}(A_{y_n}^{(-\mu)}(t)),X^{(-\mu)}_t)\/.
     \end{eqnarray}
    We consider $D=\{(y_1,\ldots,y_{n-1},y_n)\in\H^n: y_n>1\}$ and the first exit time of $Y^{(\mu)}$ from $D$ 
    \begin{eqnarray*}
      \tau_D = \inf\{t\geq 0: Y^{(\mu)}_t\notin D \} = \inf\{t\geq 0: X^{(-\mu)}_t\notin (1,\infty) \} = \tau\/,
    \end{eqnarray*}
     where $\tau=$ is the first exit time from the set $(1,\infty)$ of a geometric Brownian motion $X^{(-\mu)}$ defined in Preliminaries. We denote by $P^{(\mu)}(y,z)$, $y\in D$ and $z\in\partial D$, the Poisson kernel of $D$, i.e. the density of the distribution of $Y^{{\mu}}_{\tau_D}$, with $Y^{(\mu)}_0 = y_n>1$.
     \begin{thm}
       \label{Poisson:estimate:thm}
       For every $\mu>0$ we have
       \begin{eqnarray}
         \label{Poisson:estimate}
          P^{(\mu)}(y,z) \approx \frac{y_n-1}{|z-y|^{n}}\left(\frac{y_n} {{\cosh} d_{\H^n}(y,z)}\right)^{\mu-1/2}\/,
       \end{eqnarray}
       where $y=(\tilde{y},y_n)$, $y_n>1$ and $z=(\tilde{z},1)$, $\tilde{z}\in\R^{n-1}$. 
     \end{thm}
     \begin{proof}
          Let us denote by
     \begin{eqnarray*}
        g_t(w) = \frac{\exp(-|w|^2/2t)}{(2\pi t)^{(n-1)/2}}\/,\quad w\in\R^{n-1}
     \end{eqnarray*}
     the Brownian motion transition density in $\R^{n-1}$, $n=2,3,\ldots$. Using the fact that $\tilde{B}$ and $A^{(-\mu)}_{y_n}(\tau)$ are independent we obtain
     \begin{eqnarray*}
       P^{(\mu)}(y,z) = \int_0^\infty g_t(\tilde{z}-\tilde{y})q_{y_n}^{(-\mu)}(t) dt\/,\quad y_n>1\/, z\in\R^{n-1}\/.
     \end{eqnarray*}
       Using the estimates given in (\ref{Bessel:hittingtime:estimates}) we obtain
       \begin{eqnarray*}
          P^{(\mu)}(y,z) &\approx& \lambda \int_0^\infty e^{-(|\tilde{z}-\tilde{y}|^2+\lambda^2)/2t}\frac{y_n^{2\mu-1}}{t^{\mu-1/2}+y_n^{\mu-1/2}}\frac{dt}{t^{(n+2)/2}}\\
          &\approx& \lambda \left(y_n^{\mu-1/2}\int_0^{y_n} e^{-(|\tilde{z}-\tilde{y}|^2+\lambda^2)/2t}\frac{dt}{t^{(n+2)/2}} + y_n^{2\mu-1}\int_{y_n}^\infty e^{-(|\tilde{z}-\tilde{y}|^2+\lambda^2)/2t}\frac{dt}{t^{(n+1)/2+\mu}}\right)\\
          &=& \lambda\, y_n^{\mu-n/2-1/2}\left[\rho^{-n/2}\int_\rho^\infty u^{n/2-1}e^{-u}du+\rho^{1/2-\mu-n/2}\int_0^\rho u^{n/2-3/2+\mu}e^{-u}du\right]\/,
       \end{eqnarray*}
       where $\rho=\dfrac{|\tilde{z}-\tilde{y}|^2+\lambda^2}{2y_n}$, $\lambda=y_n-1$.
       Using (\ref{large1}) (see Appendix) we can see that 
       \begin{eqnarray*}
          P^{(\mu)}(y,z) \approx \lambda \frac{y_n^{\mu-1/2}}{(|\tilde{z}-\tilde{y}|^2+\lambda^2)^{n/2}}\/,\quad \frac{2y_n}{|\tilde{z}-\tilde{y}|^2+\lambda^2}\geq 1
       \end{eqnarray*}
       and
       \begin{eqnarray*}
       P^{(\mu)}(y,z) \approx \lambda \frac{y_n^{\mu-1/2}}{(|\tilde{z}-\tilde{y}|^2+\lambda^2)^{n/2}}\frac{(2y_n)^{\mu-1/2}}{(|\tilde{z}-\tilde{y}|^2+\lambda^2)^{\mu-1/2}}\/,\quad \frac{2y_n}{|\tilde{z}-\tilde{y}|^2+\lambda^2}< 1\/.
       \end{eqnarray*}
  Combining both estimates and using the formula for the hyperbolic distance we obtain
  \begin{eqnarray*}
      P^{(\mu)}(y,z) &\approx&  \lambda\frac{y_n^{\mu-1/2}}{(|\tilde{z}-\tilde{y}|^2+\lambda^2)^{n/2}}\left(\frac{1} {1+\frac{|\tilde{z}-\tilde{y}|^2+\lambda^2}{2y_n}}\right)^{\mu-1/2} 
        = \frac{\lambda}{|z-y|^{n}}\left(\frac{y_n} {\cosh d_{\H^n}(y,z)}\right)^{\mu-1/2}\/.
   \end{eqnarray*}
      \end{proof} 
      \begin{rem}
        The operator $\Delta_\mu$ is strongly elliptic operator on every bounded (in hyperbolic metric) subset of $\H^n$. Consequently, the hyperbolic Poisson kernels of such set are comparable with Euclidean ones. However, considered set $D$ is unbounded in $\H^n$ and the general comparison results can not be applied. {Besides}, the function $P^{(\mu)}(y,z)$ for $\mu\neq 1/2$ is no longer comparable with Euclidean Poisson kernel of upper half-space and the difference in behavior of those two functions is described by the factor
        \begin{eqnarray*}
          \left(\frac{y_n} {{\cosh} d_{\H^n}(y,z)}\right)^{\mu-1/2}\/.
        \end{eqnarray*}
      \end{rem}

 \section{APPENDIX}     
 \subsection{Uniform estimates for some class of integrals}
 \begin{lem} 
 \label{gamma}
 For $\nu\ge0$,  $0\le a<b$ and $d>0$ we have
  \begin{eqnarray}
    \label{large1} \int_{a}^b u^{\nu}e^{-du}du \stackrel c \approx 
    b^\nu\left(\frac{a+\frac 1d}{b+\frac 1d}\right)^\nu e^{-ad}\frac {b-a} {d(b-a)+1},
 \end{eqnarray}   
 where $c=c(\nu)$.
 \end{lem} 
 \begin{proof}  Let 
 $F(\nu, a,b,d)=\int_{a}^b u^{\nu}e^{-u}du$. Since $F(\nu, a,b,d)=d^{-\nu-1} F(\nu, ad,bd,1)$ it is enough to prove the lemma for $d=1$. Assume that $ b\ge 1$. Then  $$\int_{a}^b u^{\nu}e^{-u}du= e^{-a}\int_{0}^{b-a} (a+u)^{\nu}e^{-u}du \stackrel c \approx e^{-a}\int_{0}^{b-a} (a^{\nu}+u^{\nu})e^{-u}du \stackrel c \approx e^{-a} (a^{\nu}+1)((b-a)\wedge1),$$ which is an equivalent form of (\ref{large1}) in the case $b\ge d=1$. If $b<1$ then
  $$\int_{a}^b u^{\nu}e^{-u}du \stackrel c \approx   b^{\nu}(b-a),$$
  which is exactly (\ref{large1}) in the case $b<d=1$. Note that in all comparisons above the constant $c$ is dependent only  on $\nu$.
 
  \end{proof}
   \begin{lem}  
 \label{logestimate:lemma}
 
Let $0\le a\le 1$. Then for every $v>0$ we have

\begin{eqnarray}\int_0^a \frac{e^{-vu}u du}{\log^2 u+1}&\approx& \frac{1}{(v+1/a)^2(\log^2 (v+1/a)+1)}.\label{I}\end{eqnarray}

If additionally $av\le 1$ then 

\begin{eqnarray}\int_a^1 \frac{e^{-vu}u du}{1-\log u}\approx \frac{1-a}{(v+1)^{3/2}(1+\log (v+1))}.\label{J}\end{eqnarray}

\end{lem}  
\begin{proof}
  Let  $$J(a,v)=\int_0^a \frac{e^{-vu}u du}{\log^2 u+1}.$$ 
  
  First, assume that $av<2$ then 
  $$J(a,v)\approx \int_0^a \frac{ u du}{\log^2 u+1}\le \frac{ a^2}{2(\log^2 a+1)}.$$ 
  If additionally $a>1/2$ then  $$J(a,v)\approx 1.$$ If $a\le 1/2$ then 
  $$J(a,v)\ge e^{-2} \int_{a/2}^a \frac{u du}{\log^2 u+1}\ge c \frac{ a^2}{\log^2 a+1},$$ 
  which ends the proof (\ref{I}) in the case $av<2$. 
    
  We assume now that  $av\ge 2$.  
   Observe that for every $q ,r \in\R$, from the fact that \\$(qr/\sqrt{r^2+1}-\sqrt{r^2+1})^2\geq 0$ we get
 \begin{eqnarray*}
    \label{inequality:01}
    (q-r)^2+1\geq \frac{q^2}{r^2+1}\/.
 \end{eqnarray*}  
 This implies that $$ \frac{1}{(\log v-\log s)^2+1}\leq \frac{\log^2 s+1}{\log^2 v}  $$
   Consequently, we obtain
 \begin{eqnarray*}
   J(a,v) = \frac{1}{v^2}\int_0^{av} \frac{e^{-s}s\,ds}{(\log v-\log s)^2+1}\leq \frac{1}{v^2\log^2 v}\int_0^\infty e^{-s}s(\log^2 s+1)ds\/.
 \end{eqnarray*}
 On the other hand we get
 \begin{eqnarray*}
 J(a,v) = \frac{1}{v^2}\int_0^{av} \frac{e^{-s}s\,ds}{(\log v-\log s)^2+1}\ge
     \frac{1}{v^2(\log^2v+1)}\int_1^2 e^{-s}s\,ds,
 \end{eqnarray*}
 which ends the proof (\ref{I}).
 
 Let \begin{eqnarray*}
  J(a,v)&=&\int_{a}^1 \frac{e^{-vu}u^{1/2}}{1-\log u}\,du \\
   &=& \frac{1}{v^{3/2}}\int_{va}^v \frac{e^{-s}s^{1/2}}{1-\log s+\log v}\,ds\/.
 \end{eqnarray*}
 
 If $v<2$ then
 
 \begin{eqnarray}
 J(a,v)&\approx&\int_{a}^1 \frac{u^{1/2}}{1-\log u}\,du \approx 1-a,\label{J1}
 \end{eqnarray}
 which is (\ref{J}) in this case.

 Next, we assume that $2\le v\le 1/a$. 
 Using the fact that for every $r>q>0$ we have $r-q+1\geq r/(q+1)$ we get
 \begin{eqnarray*}
    \int_{va}^v \frac{e^{-s}s^{1/2}}{1-\log s+\log v}\,ds&\leq& \left(\int_0^1+\int_1^v\right)\frac{e^{-s}s^{1/2}}{1-\log s+\log v}\,ds\\
    &\leq& \frac{c}{\log v}\int_0^1 e^{-s}s^{1/2}ds + \frac{1}{\log v}\int_1^v  e^{-s}s^{1/2}(1+\log s)ds\leq \frac{c_2}{\log v}\/.
 \end{eqnarray*}
 Moreover, for $2<v<1/a$ we get
 \begin{eqnarray*}
    \int_{va}^v \frac{e^{-s}s^{1/2}}{1-\log s+\log v}\,ds&\geq& \int_1^2 \frac{e^{-s}s^{1/2}}{1-\log s+\log v}\,ds\geq \frac{c_1}{\log v}\/.
 \end{eqnarray*} 
 We have just proved that 
 \begin{eqnarray}
    J(a,v)\approx  \frac{1}{v^{3/2}\log v},\ 2<v<1/a.\label{J2}
 \end{eqnarray} Combining (\ref{J1}) and (\ref{J2}) we complete the proof of (\ref{J}).
\end{proof}
  
\subsection{Estimates of $w_{1,\,\lambda}$ and $w_{2,\,\lambda}$}

 \begin{lem}
   \label{w1:estimate:lemma}
   There exist constants $c=c(\mu)>0$ and $\theta_\mu>0$ such that
   \begin{eqnarray}
      \label{w1:estimate:abs}
      |w_{1,\,\lambda}(v)|\leq c x^{\mu-3/2}e^{-v\theta_\mu}\/,\quad v>0\/.
   \end{eqnarray}
 \end{lem}
\begin{proof}
Recall that the set of zeros of the function $K_\mu(z)$ is denoted  by $Z=\{z_1,...,z_{k_\mu}\}$. 
For every $z_i\in Z$ we have $\Re z_i<0$ and the set $Z$ is finite.  Consequently, there exists constant $c_1>0$ such that for every $i=1,\ldots,k_\mu$ we have
 $$\left|\frac{K_\mu(xz_i) }{ K_{\mu-1}(z_i)}\right|= \left|\frac{{K_\mu(xz_i)-K_\mu(z_i)} }{ K_{\mu-1}(z_i)}\right|\le c_1(x-1).$$
Moreover, using (\ref{asymp_K_infty}), the constant $c_1$ can be chosen to ensure that for every $x\ge 2$ we have

 $$\left|\frac{e^{\lambda z_i}K_\mu(xz_i) }{ K_{\mu-1}(z_i)}\right|\le c_1 \frac1{\sqrt{x}}.$$
 Hence there are $c_2>0$ and $\theta_\mu= -\max_i\{\Re z_i\}>0$ such that
 $$\left|\frac{z_i e^{\lambda z_i} K_\mu(xz_i)}{ K_{\mu-1}(z_i)} \/ e^{z_i v}\right|\le c_2\frac \lambda{x^{3/2}}e^{-\theta_\mu v}$$ 
 and it gives  $$ \left|w_{1,\/\lambda}(v)\right|\le  c_3 x^{\mu-3/2} e^{-\theta_\mu v}\/.$$ 

%
%
\end{proof}
  Let us define for $u>0$ and $x>1$:
   \begin{eqnarray*}
    S_\mu(x,u)=I_\mu\(xu\)K_\mu(u)-I_\mu(u)K_\mu\(xu\)\/.
   \end{eqnarray*}
The function $S_\mu$ appears in the formula for $w_{2,\lambda}$ and consequently, the uniform estimates of the function $S_\mu$, which are given in the next Lemma, are crucial to get the estimates of the function $w_{2,\lambda}$ given in Lemma \ref{w2:estimate:lemma} for $\mu>0$ and in Lemma \ref{w:muzero:estimate:lemma} in the case $\mu=0$.

 \begin{lem}
     \label{S1:estimate:lemma}
     For $\mu\geq 0$ we have
   \begin{eqnarray}
      \label{S1:estimate:1}
   \frac{\lambda}{x}\frac {K_{\mu}(xu)}{K_{\mu}(u)}   \le  S_\mu(x,u)\le \lambda\frac{K_{\mu}(u)}{K_{\mu}(xu)}.
   \end{eqnarray}   
  There are constants $c_1$ and $c_2$ such that for $\mu\ge 0$, $1<x<2$ and $u>0$ we thus obtain
  \begin{eqnarray}
      \label{S1:estimate:2}
 c_1 \lambda \, e^{-\lambda\,u} \leq  S_\mu(x,u)\leq c_2 {\lambda}\, e^{\lambda\,u}\,.
  \end{eqnarray}  
  There is a  constant $c_1$ such that for $\mu>0$, $x>2$ and $u>0$ we have 
  \begin{eqnarray}
   \label{S1:estimate:3}
 c_1I_\mu\(xu\)K_\mu(u)  \leq  S_\mu(x,u)\leq I_\mu\(xu\)K_\mu(u)\,.
  \end{eqnarray}  
  There is a  constant $c_1$ such that for  $\mu=0$, $x>2$ and $xu>1$ we have 
  \begin{eqnarray}
   \label{S1:estimate:4}
 c_1I_0\(xu\)K_0(u)  \leq  S_0(x,u)\leq I_0\(xu\)K_0(u)\,.
  \end{eqnarray}
    For $\mu=0$, $x>2$ and $xu<1$ we have 
  \begin{eqnarray}
   \label{S1:estimate:4b}
 S_0(x,u)\approx \log x.
  \end{eqnarray}
  
  \end{lem}
  
   \begin{proof}
 Write for $u>0$
 \begin{eqnarray*}
 \psi(u) = \frac{I_{\mu}(u)}{K_{\mu}(u)}\,. 
 \end{eqnarray*}
 Then by (\ref{Wronskian}) we have
 \begin{eqnarray*}
 \psi'(u) = \frac{1}{u}\, \frac{1}{K_{\mu}^2(u)}\,. 
 \end{eqnarray*}
 Writing 
  \begin{eqnarray*}
 S_{\mu}(x,u) = \[\psi(xu)-\psi(u)\]\,K_{\mu}(xu)\,K_{\mu}(u) \,. 
 \end{eqnarray*}
 we obtain from the Lagrange theorem
  \begin{eqnarray*}
 S_{\mu}(x,u) = \frac{xu-u}{\theta xu}\, \frac{K_{\mu}(xu)\,K_{\mu}(u)}{K_{\mu}(\theta xu)\,K_{\mu}(\theta xu)} \,. 
 \end{eqnarray*}
The quantity $\theta$ here has the property $1\leq \theta x \leq x$.
This and the monotonicity of the function $K_{\mu}$,  give the estimate (\ref{S1:estimate:1}). The estimate (\ref{S1:estimate:2})  is a direct consequence of the  limiting behaviour of the function $K_{\mu}$.

 To prove (\ref{S1:estimate:3}) and (\ref{S1:estimate:4}) note that  
the function  $g(x,u)=\frac{\psi(u)}{ \psi(xu)}= \frac{ I_\mu\(u\)K_\mu(xu)}{K_\mu\(u\)I_\mu\(xu\)}$  as a function of $x$ is decreasing. Hence for $ x>2$,  $$g(x,u)\le g(2,u)<1, u>0. $$
If $\mu>0$, then the limits at $0$ and $\infty$ of $g(2,u)$ are strictly less then $1$.  By continuity 
  $$\sup_{u>0}g(2,u)= a<1.$$
  If $\mu=0$ by the same argument for $ x>2$,  $$g(x,u)\le \sup_{v\ge 1/2}g(2,v)<1, u\ge 1/2. $$
  If $xu>1$, $u<1/2$ and $x>2$ then
   $$g(x,u)= \frac{ I_0\(u\)K_0(xu)}{K_0\(u\)I_0\(xu\)}\le \frac{ I_0\(1/2\)K_0(1)}{K_0\(1/2\)I_0\(1\)}= g(2,1/2) <1.$$
   These estimates  imply that 
   $$g(x,u)\le \sup_{v\ge 1/2}g(2,v)=a<1, u>1/x. $$

  %
  Hence,  in both cases ($\mu=0$ or $\mu>0$) we have for $x>2$  and $xu>1$
$$
 S_{\mu}(x,u) =   \[1-\frac {\psi(u)}{\psi(xu)}\]\,\psi(xu) K_{\mu}(xu)\,K_{\mu}(u)\ge (1-a)I_{\mu}(xu)\,K_{\mu}(u)\,,
$$
  which  ends the proof of (\ref{S1:estimate:3}) and (\ref{S1:estimate:4}).

 It remains to consider $\mu=0, x>2$ and 
     $xu<1$. We apply the asymptotics of $K_0$ at $0$. Namely, by (\ref{K0_atzero}) we can write 
   $$K_0(z)= -\log \frac z2 \, I_0(z)+ A(z), $$where  $A(z)\to c>0, z\to 0$.  This yields 
   $$S_{0}(x,u)= I_0(ux)I_0(u)\log x + I_0(ux)A(u)- I_0(u)A(ux)\approx \log x.$$
  \end{proof}

 \begin{lem} 
 \label{w2:estimate:lemma}
 For $\mu>0$ and  $x>1$ we have
 \begin{eqnarray}
    \label{w2:estimates}
    w_{2,\/\lambda}(v)\approx  {(-\cos(\pi\mu))}\frac {x^{2\mu-1}}{(v+1)^{\mu+3/2}(v+x)^{\mu+1/2}}\/,\  v>0 \/.
 \end{eqnarray}   
 \end{lem}
 
 \begin{proof}
  Let us denote  
   \begin{eqnarray*}
   h(x,u,v)&=& \frac{ S_\mu\(x,u\)}{
\cos^2(\pi\mu) K_\mu^2(u)+(\pi I_\mu(u)+\sin(\pi \mu) K_\mu(u))^2} \/
 e^{-\lambda u} e^{-vu} \/u\\
 &\approx& \frac{ S_\mu\(x,u\)}{
 K_\mu^2(u)+I_\mu^2(u)} \/
 e^{-\lambda u} e^{-vu} \/u.
 \end{eqnarray*}
 Since 
 $w_{2,\/\lambda}(v)= 
-\cos(\pi\mu) \frac{x^\mu }{ \lambda}
 \int_0^\infty h(x,u,v)  du$, it is enough to estimate  $\int_0^\infty h(x,u,v) du$.
  which is done below for two cases.

A)  Case $1<x<2$.\newline
 Suppose that $ 0<u<1$.  By Lemma \ref{S1:estimate:lemma},
 $S_\mu\(x,u\)\approx \lambda$, thus
 \begin{eqnarray*}h(x,u,v)
 &\approx& \frac{\lambda  }{
 K_\mu^2(u)+I_\mu^2(u)} \/
 e^{-\lambda u} e^{-vu} \/u\\ &\approx&  \lambda u^{2\mu+1}e^{-(v+1)u},\ 0<u<1. \end{eqnarray*}
 For $u>1$ we have, by Lemma \ref{S1:estimate:lemma}, $S_\mu\(x,u\)\le  c\lambda e^u$, which yields the following upper bound.
 \begin{eqnarray*}h(x,u,v)
 &\le& c \frac{\lambda e^u }{
 K_\mu^2(u)+I_\mu^2(u)} \/
 e^{-\lambda u} e^{-vu} \/u\\ &\le&  c\lambda u\/e^{-(v+1)u}. \end{eqnarray*}
 Applying Lemma \ref{gamma} to  the above estimates we arrive  at
 $$\int_0^1 h(x,u,v)du\approx \frac {\lambda} {(v+1)^{2\mu+2}}$$
and 
 $$\int_1^\infty h(x,u,v)du\le c \frac {\lambda}{(v+1)^2}e^{-(v+1)}.$$
 Combining both integrals we obtain  
 $$\int_0^\infty h(x,u,v)du\approx \frac {\lambda} {(v+1)^{2\mu+2}},$$ 
 which proves the lemma in the case $1<x<2$.
 
B) Case $x>2$. \newline
 By  Lemma \ref{S1:estimate:lemma}, $ S_\mu(x,u)\approx I_\mu(xu)K_\mu\(u\)$, which implies
 \begin{eqnarray*} h(x,u,v)
 &\approx& \frac{ I_\mu(xu)K_\mu\(u\)}{K_\mu^2(u)+I^2_\mu(u)} \/
 e^{-\lambda u} e^{-vu} \/u.
 \end{eqnarray*}
 Next, using the asymptotics of the Bessel functions, we arrive at
  $$h(x,u,v)
\approx\left\{
\begin{array}{ll}
 x^\mu u^{2\mu+1}e^{-(v+\lambda)u}, & \hbox{$xu<1$,} \\
    x^{-1/2}u^{\mu+1/2}e^{-vu}, & \hbox{$1/x<u<1$,}\\
   x^{-1/2}ue^{-(v+2)u}, &\hbox{$u>1$.}\end{array}
\right.$$
To estimate $H(x,v)=\int_0^\infty h(x,u,v)du$ we split the integral into three parts:
 \begin{eqnarray*}  
 H(x,v)&=&\int_0^{1/x} h(x,u,v)du+\int_{1/x}^1 h(x,u,v)du+\int_{1}^\infty h(x,u,v)\\
      &=&J_1(x,v)+J_2(x,v)+J_3(x,v).
 \end{eqnarray*}
 Applying  Lemma \ref{gamma} with $a=0, b=1/x$ and $d= v+\lambda$, the first integral can be estimated in the following way:
 \begin{eqnarray*}J_1(x,v)&\approx& x^\mu \int_0^{1/x}u^{2\mu+1}e^{-(v+\lambda)u}du\\&\approx& \frac{x^\mu}{(v+\lambda)^{2\mu+1}} \frac  {1/x}{1+(v+\lambda)/x} \\
 &\approx& \frac{ x^{\mu}}{(v+x)^{2\mu + 2}}.\end{eqnarray*}
 Next, we deal with the second integral.  Again, by Lemma  \ref{gamma} with $a=1/x, b=1$ and $d= v$, we obtain
 %
 %
  \begin{eqnarray*} J_2(x,v) &\approx& x^{-1/2} \int_{1/x}^1 u^{\mu+1/2}e^{-vu}du\\
 &\approx& x^{-1/2}\left(\frac{\frac1 x +\frac1 v}{1+\frac1 v}\right)^{\mu+1/2} e^{-v/x}\frac {1-1/x}{1+v-v/x}\\ &\approx& x^{-1/2}\left(\frac{\frac1 x +\frac1 v}{1+\frac1 v}\right)^{\mu+1/2} e^{-v/x}\frac {\lambda }{x+\lambda v}\\ &\approx& x^{-1/2}\left(\frac{1+\frac v x }{1+ v}\right)^{\mu+1/2} e^{-v/x}\frac {1 }{ v+1}.
  \end{eqnarray*}
  The third integral can be estimated for $v>0$ as follows.
$$ J_3(x,v) \approx  x^{-1/2}  \int_{1}^\infty ue^{-(v+2)u}du \approx x^{-1/2}  \frac 1{(v+2)}  e^{-(v+2)}. $$
It is clear that $$ J_3(x,v) \le c \frac {x^{\mu}}{(v+1)^{\mu+3/2}(v+x)^{\mu+1/2}}.$$
Next, 
 $$\frac {(v+1)^{\mu+3/2}(v+x)^{\mu+1/2}} {x^{\mu}} I_1\approx \frac {(v+1)^{\mu+3/2}(v+x)^{\mu+1/2}} {x^{\mu}} \frac{ x^{\mu}}{(v+x)^{2\mu + 2}} = \frac {(v+1)^{\mu+3/2}}{(v+x)^{\mu+3/2}}$$
  and
 \begin{eqnarray*}\frac {(v+1)^{\mu+3/2}(v+x)^{\mu+1/2}} {x^{\mu}} I_2&\approx& \frac {(v+1)^{\mu+3/2}(v+x)^{\mu+1/2}} {x^{\mu}}  x^{-1/2}\left(\frac{1+\frac v x }{1+ v}\right)^{\mu+1/2} e^{-v/x}\frac {1 }{ v+1} \\
 &=& (1+ v/x)^{2\mu+1} e^{-v/x}.\end{eqnarray*}
 The observation   $\frac {(v+1)^{\mu+3/2}}{(v+x)^{\mu+3/2}}+ (1+ v/x)^{2\mu+1} e^{-v/x}\approx 1$ completes the proof.

 \end{proof}
  \begin{lem}
\label{w:muzero:estimate:lemma}
For $\mu=0$ and  $x>1$ we have
 
\begin{eqnarray*}
   -w_\lambda(v) &\approx& \frac{1}{x(v+1)^{3/2}(v+x)^{1/2}}\frac{\log (x+1)}{\log(v+2)(\log(x+1)+\log(v+2))}\/,\ v>0\/.
\end{eqnarray*}
\end{lem}  
\begin{proof}
Let us denote 
\begin{eqnarray*}
   f(x,v) =  \frac{1}{x(v+1)^{3/2}(v+x)^{1/2}}\frac{\log (x+1)}{\log(v+2)(\log(x+1)+\log(v+2))}\/.
\end{eqnarray*}
  We write
\begin{eqnarray*}
   \int_0^\infty \frac{S_0(x,u)\,e^{-\lambda u}}{K_0^2(u)+\pi^2I_0^2(u)}\,e^{-vu} udu &=& \left(\int_0^{1/x}+\int_{1/x}^1+\int_1^\infty\right) \frac{S_0(x,u)\,e^{-\lambda u}}{K_0^2(u)+\pi^2I_0^2(u)}\,e^{-vu} udu\\
   &=& J_1(x,v)+J_2(x,v)+J_3(x,v).
\end{eqnarray*}
 The estimates of $J_3(x,v)$ are exactly the same as the corresponding estimates proved in Lemma \ref{w2:estimate:lemma}. 
 Hence,  for $x\geq 2$, 
 \begin{eqnarray}
   J_3(x,v) &\approx& 
    \frac{1}{\sqrt{x}}\frac{e^{-(v+2)}}{(v+2)^2}\/,\quad v>0\/,\label{I31}
 \end{eqnarray}
and  for $1<x<2,$ 
 \begin{eqnarray} %
   J_3(x,v) &\le & c\frac{\lambda}{(v+1)^2}e^{-(v+1)}.\label{I32}
 \end{eqnarray}
 
 To estimate $J_2(x,v)$ for $x>2$ observe that for $xu\geq 1$ and $u<1$, $S_0(x,u)\approx \dfrac{e^{xu}}{\sqrt{xu}}K_0(u)\approx \dfrac{e^{xu}}{\sqrt{xu}}(1-\log u)$. This follows from    the asymptotic expansions for $I_0$ and $K_0$ (see (\ref{I_atzero}) and (\ref{K_atzero})), and  Lemma \ref{S1:estimate:lemma}. Thus
 \begin{eqnarray}
   J_2(x,v) &\approx& \frac{1}{x^{1/2}}\int_{1/x}^1 \frac{K_0(u)e^{u}e^{-vu}u^{1/2}}{K_0^2(u)+\pi^2I_0^2(u)}\,du \approx  \frac{1}{x^{1/2}}\int_{1/x}^1 \frac{e^{-vu}u^{1/2}}{1-\log u}\,du \nonumber \\
&\approx& \frac{1}{x^{1/2}}\frac{1-1/x}{(v+1)^{3/2}(1+\log (v+1))},\ v\le x,  \label{I21}
 \end{eqnarray}
 where the last step is a consequence of (\ref{J}) with $a=1/x$. Next, 
 
  \begin{eqnarray}
   J_2(x,v) &\approx& \frac{1}{x^{1/2}}\int_{1/x}^1 \frac{K_0(u)e^{u}e^{-vu}u^{1/2}}{K_0^2(u)+\pi^2I_0^2(u)}\,du \le  c\frac{1}{x^{1/2}}\int_{0}^1 \frac{e^{-vu}u^{1/2}}{1-\log u}\,du \nonumber \\
&\le& c\frac{1}{x^{1/2}}\frac{1}{(v+1)^{3/2}(1+\log (v+1))},\ v>0,  \label{I22}
 \end{eqnarray}
  where, again,  the last step is a consequence of (\ref{J}) with $a=0$.
 
 For $x<2$ and  $1/x\le u<1$, by Lemma \ref{S1:estimate:lemma}, we have $S_0(x,u)\approx \lambda$ hence 
 
 \begin{eqnarray}
   J_2(x,v) &\approx& \lambda \int_{1/x}^1 \frac{e^{-vu}u}{K_0^2(u)+\pi^2I_0^2(u)}\,du \approx \lambda \int_{1/x}^1 \frac{e^{-vu}u}{1+\log^2 u}\,du. \label{I24} 
 \end{eqnarray}

 Finally, for $0<u<1/x$ we have, by Lemma \ref{S1:estimate:lemma}, $S_0(x,u)\approx \log x,\ x>0$   and consequently using (\ref{I}), with $a=1/x$, we obtain
 \begin{eqnarray}
    J_1(x,v) &\approx& \log x\int_0^{1/x} \frac{e^{-\lambda u}e^{-vu}u}{K_0^2(u)+\pi^2I_0^2(u)}\,du \approx \log x\int_0^{1/x} \frac{e^{-vu}u}{1+\log^2 u}\,du \label{I11}.\\
   &\approx&  \frac{\log x}{(v+x)^2(\log^2 (v+x)+1)}.\label{I12}
 \end{eqnarray}

 We are now ready to estimate the function $-w_\lambda(v)$. At first  we  consider $1<x<2$. Taking into account that $\log x\approx \lambda$, using  (\ref{I24}) and (\ref{I11}) w arrive at 
 \begin{eqnarray}
  J_1(x,v)+ J_2(x,v) &\approx& \lambda \int_{0}^1 \frac{e^{-vu}u}{1+\log^2 u}\,du \\
&\approx& \frac{\lambda}{(v+1)^2(\log^2 (v+1)+1)}. 
  \end{eqnarray}
 Combining this with (\ref{I32}) we have for $1<x<2$,
 
 \begin{eqnarray*}
    -w_\lambda(v) &=& \frac{1}{\lambda}(J_1(x,v)+ J_2(x,v)+ J_3(x,v))  \approx \frac{1}{\lambda}(J_1(x,v)+ J_2(x,v))\\
     &\approx& \frac{1}{(v+1)^2(\log^2 (v+1)+1)} \approx  f(x,v)\/.
 \end{eqnarray*}

 Taking into account (\ref{I31}), (\ref{I21}) and (\ref{I12}), we infer that for $0<v<2$ and $x>2$, $J_1(x,v)+ J_2(x,v)\le c J_3(x,v),$ which yields
 \begin{eqnarray*}
    -w_\lambda(v) = \frac{1}{\lambda}(J_1(x,v)+ J_2(x,v)+ J_3(x,v))  \approx \frac{1}{\lambda}J_3(x,v) \approx \frac{1}{x^{3/2}}\approx  f(x,v)\/.
 \end{eqnarray*}
 For $2 \leq v<x$, by (\ref{I12}), (\ref{I31}) and (\ref{I21}), $J_1(x,v)+ J_3(x,v)\le c J_2(x,v)$. Hence,
 \begin{eqnarray*}
    -w_\lambda(v) = \frac{1}{\lambda}(J_1(x,v)+ J_2(x,v)+ J_3(x,v)) \approx \frac{1}{\lambda} J_2(x,v) \approx \frac{1}{x^{3/2}v^{3/2}\log v}\approx  f(x,v)\/.
 \end{eqnarray*}
 Finally for $v\ge x>2$, by (\ref{I22}), (\ref{I12}) and (\ref{I31}) we have $J_2(x,v)+J_3(x,v)\leq c J_1(x,v)$ for some constant $c>1$ and consequently
 \begin{eqnarray*}
    -w_\lambda(v) = \frac{1}{\lambda}(J_1(x,v)+ J_2(x,v)+ J_3(x,v)) \approx \frac{1}{\lambda} J_1(x,v) \approx \frac{\log x}{xv^2\log^2v}\approx f(x,v).
 \end{eqnarray*}
  The proof is completed.
\end{proof} 

 \bibliography{bibliography}
\bibliographystyle{plain}

\end{document}